\def\dofigures{11}
\documentclass{article}

\usepackage{calc}

\usepackage{amsmath}
\usepackage{amsxtra}
\usepackage{amssymb}
\usepackage{amsfonts}
\usepackage{amsthm}
\usepackage{amstext}
\usepackage{amsbsy}
\usepackage{amscd}
\usepackage[sans]{dsfont}

\usepackage{color}
\usepackage{bbding}

\newcommand{\R}{\mathds{R}}

\newcommand{\defm}[1]{\emph{#1}}
\newcommand{\subeq}[2]{\mathord{\underbrace{\mathop{#1}}_{#2}}}

\newcommand{\sgn}{\operatorname{sgn}}

\newcommand{\supp}{\operatorname{supp}}

\newcommand{\loc}{{\operatorname{loc}}}

\theoremstyle{plain}
\newtheorem{theorem}{Theorem}
\newtheorem{lemma}{Lemma}

\theoremstyle{definition}

\newtheorem{definition}[lemma]{Definition}
\theoremstyle{remark}
\newtheorem{remark}[lemma]{Remark}

\def\XXint#1#2#3{{\setbox0=\hbox{$#1{#2#3}{\int}$}
\vcenter{\hbox{$#2#3$}}\kern-.5\wd0}}

\newcommand{\BV}{\operatorname{BV}}

\newcommand{\aint}{{-\kern-5.3mm\int}}

\newcommand{\boi}[2]{{]#1,#2[}}
\newcommand{\loi}[2]{{]#1,#2]}}
\newcommand{\roi}[2]{{[#1,#2[}}
\newcommand{\cli}[2]{{[#1,#2]}}

\newcommand{\topref}[2]{\overset{\text{\eqref{#1}}}{#2}}

\newcommand{\myeqref}[1]{\eqref{#1}}
\newcommand{\myref}[1]{\ref{#1}}
\newcommand{\mylabel}[1]{\label{#1}}
\newcommand{\myeqlabel}[1]{\label{#1}}

\usepackage{graphicx}

\usepackage{bbding}
\usepackage{perpage}
\usepackage{mathtools}

\usepackage{calc}

\usepackage{amsmath}
\usepackage{amsxtra}
\usepackage{amssymb}
\usepackage{amsfonts}

\usepackage[sans]{dsfont}

\usepackage{color}
\definecolor{lightgrey}{rgb}{0.7,0.7,0.7}

\usepackage{textcomp}
\usepackage{mathrsfs}

\newcommand{\ignoretag}[1]{}

\newcommand{\Ph}{\mathcal{P}}

\newcommand{\xb}{\xi_0}
\newcommand{\Ub}{\overline{U}}
\newcommand{\Vb}{\overline{V}}

\newcommand{\bO}{\mathcal{O}}
\newcommand{\Ps}{\mathcal{P}}
\newcommand{\Peps}{\Ps_{\ueps}}
\newcommand{\ueps}{\epsilon}
\newcommand{\RP}{\mathds{RP}}
\newcommand{\Sx}{\mathcal{S}}
\newcommand{\Rx}{\mathcal{R}}
\newcommand{\Cx}{\mathcal{C}}

\title{Steady and self-similar inviscid flow\footnote{This material is based upon work supported by the
National Science Foundation under Grant No.\ NSF DMS-0907074}}
\author{Volker Elling and Joseph Roberts}

\begin{document}

\maketitle

\begin{abstract}
    We consider solutions of the 2-d compressible (isentropic) Euler equations that are steady and self-similar.
    They arise naturally at interaction points in genuinely multi-dimensional flow. 
    We characterize the possible solutions in the class of flows $L^\infty$-close to a constant supersonic background.
    As a special case we prove that solutions of 1-d Riemann problems are unique in the class of small $L^\infty$ functions.
    We also show that solutions of the backward-in-time Riemann problem are necessarily $\BV$.
\end{abstract}

\section{Introduction}

We consider systems of hyperbolic conservation laws in two dimensions:
\begin{alignat}{1}
    U_t + f^x(U)_x + f^y(U)_y &= 0. \notag
\end{alignat}
Most important are the 2-d \defm{compressible Euler equations} for motion of inviscid fluids: 
$U=(\rho,\rho v^x,\rho v^y)$ ($\rho$ density, $\vec v$ velocity) with fluxes
\begin{alignat}{1} &
    f^x(U) = v^xU + \begin{bmatrix}
        0 \\
        p \\
        0 
    \end{bmatrix}, \quad f^y(U) = v^yU + \begin{bmatrix}
        0 \\
        0 \\
        p 
    \end{bmatrix}.
\notag\end{alignat}
where $p=p(\rho)$ is pressure. 

Our aim is to increase understanding of genuinely multi-dimensional flow, in particular its wave interactions. 
Some examples are regular reflection (four shock waves meeting at a point) 
\cite{chen-feldman-selfsim-journal,elling-liu-pmeyer,canic-keyfitz-lieberman,yuxi-zheng-rref,henderson-etal,elling-rrefl-lax,elling-sonic-potf,elling-detachment}
or Mach reflection (three shocks meeting with a contact or another type of wave) 
\cite{ben-dor-book,ben-dor-shockwaves2006,hornung-reviews,hunter-tesdall,vasiliev-kraiko,skews-1997}.
In these cases there are distinguished points near which the flow is, to leading order, \emph{constant along rays} starting
in the point.
This leads to solutions that are self-similar and steady (from the point of view of an observer moving with the interaction point):
$$U(t,x,y) = U(\phi), \quad \phi=\measuredangle(x,y)\in\roi{0}{2\pi}. $$
In the case of Mach reflection the precise nature of the interaction remains controversial after decades of research. 
It is known that \defm{triple points} (three shocks, with smooth flow in between) are not possible in most reasonable
models (see \cite{neumann-1943}, \cite[Section 129]{courant-friedrichs}, \cite{henderson-menikoff}, \cite[Theorem 2.3]{serre-hbfluidmech}). 
However, beyond results for triple points and other special cases, 
the possible combinations have apparently never been classified systematically.
Such a classification is our ultimate goal.

We are particularly motivated by an example in \cite{elling-nuq-journal} which features a steady and self-similar
solution where two shocks and two contacts meet in a point. Numerical calculations suggest there is a second unsteady solution,
with the steady one as initial data, so that the Cauchy problem for the 2-d Euler equations would not have uniqueness,
at least in its current formulation. Naturally we wonder which other steady self-similar solutions exhibit this behaviour
and what characterizes them.

The literature on \defm{multi-dimensional Riemann problems} \cite{li-zhang-yang,zhang:593,yuxi-zheng-book,lax-xdliu} 
is somewhat related to our flow class. 
However, in those problems only the \emph{initial data} is necessarily constant along rays; we are interested
in the special case where the forward-in-time solution equals the initial data. On the other hand, much of that literature
focuses on initial data constant in each \emph{quadrant}, 
a setting that is apparently so restrictive that the numerical studies have not encountered
non-uniqueness phenomena like those observed in \cite{elling-nuq-journal}.

In this article we focus on the case where $U$ is a small (in $L^\infty$) perturbation 
of a constant \emph{supersonic} background state $\Ub$. 
Interestingly we do \emph{not} need to assume that $U$ is in $\BV$, the space of functions of bounded variation;
instead we will \emph{prove} it (under standard assumptions about $p(\rho)$).
This is crucial because in several space dimensions $\BV$ is probably too narrow to contain all reasonable flows \cite{rauch-bv},
in contrast to one space dimension where a satisfactory theory has been based on $\BV$  or closely related classes
\cite{glimm,glimm-lax,bianchini-bressan-j}.

Our results also apply to the classical case of 1-d Riemann problems for strictly hyperbolic conservation laws 
whose eigenvalues are either genuinely nonlinear or linearly degenerate: 
for sufficiently small jump, their self-similar forward solutions (see \cite{lax-claws-ii} or \cite[Chapter 11]{evans} for construction) 
are unique in the class of $L^\infty$ (rather than $\BV$) solutions with small norm 
(for related uniqueness results see \cite[Section 9.1]{dafermos-book} and 
\cite{dafermos-sbv-2008,bressan-goatin,bressan-lefloch,bressan-crasta-piccoli,liu-yang,oleinik-1959,kruzkov,smoller-uq-1969}).
This generalizes an earlier result of Heibig \cite{heibig} which required all eigenvalues to be genuinely nonlinear.
While uniqueness need not hold backward in time, we are able to show that small-$L^\infty$ solutions must be small-$\BV$
(which cannot be improved to any smaller commonly used class since examples with infinitely many jumps are easy to construct).

We now summarize our main result.  Consider the 2-d compressible isentropic Euler equations.   Let $U \in L^{\infty}$ be a steady, self-similar, entropy-admissible weak solution, with $||U-\Ub||_{L^{\infty}}< \epsilon$ for some supersonic background state $\Ub$ and $\epsilon > 0$.  If $\epsilon$ is sufficiently small, then $U \in BV$ and it must have the structure shown in Figure \myref{fig:test}:

\if\dofigures%
\begin{figure}[h]
\begin{center}
    \input{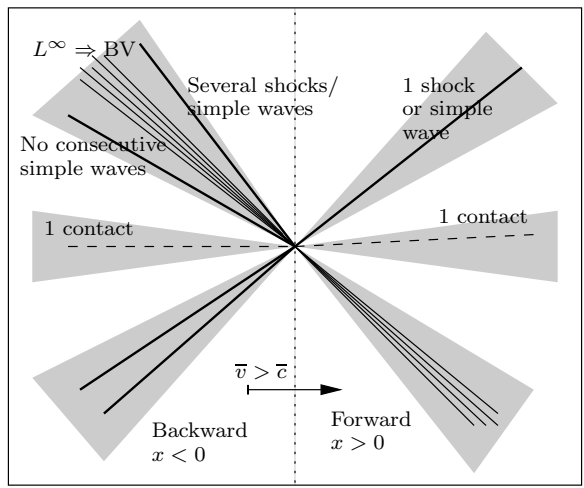}
\caption{$U$ must be constant outside narrow sectors specified by eigenvalues evaluated at $\Ub$. 
Linearly degenerate sectors: at most one contact discontinuity. 
Genuinely nonlinear forward sectors: at most one shock or simple wave.
Genuinely nonlinear backward sectors: infinitely many waves possible, but no
consecutive simple waves.  Here we have taken the background state to have horizontal velocity $(\overline{v},0)$ and sound speed $\overline{c}$.}
\mylabel{fig:test}
\end{center}
\end{figure}
\fi%

\section{Balance laws}

Let $\Ps \subset \R^m$ be an open set. 
Consider smooth functions $\eta,\psi^x,\psi^y:\Ps\rightarrow\R$.
For $A\subset\R^3$ we say $U=(U^1,...,U^m)\in L^1_{\loc}(\R^3;\Ps)$ is a \defm{weak solution} of
\begin{alignat}{1}
    \eta(U)_t + \psi^x(U)_x + \psi^y(U)_y &\leq 0 \qquad\text{in $A$} \myeqlabel{eq:entropy}
\end{alignat}
if the inequality is satisfied in the \defm{weak sense} (or: \defm{distributional sense}):
every $x\in A$ has an open neighbourhood $N$ so that for \emph{nonnegative} smooth $\Phi$ with $\supp\Phi\Subset N$,
\begin{alignat}{1}
    - \int_{\R^3} \Phi_t \eta(U) + \Phi_x \psi^x(U) + \Phi_y \psi^y(U) d(x,y,t) &\leq 0 \myeqlabel{eq:weak-entropy}
\end{alignat}
We call $U$ a \defm{strong solution} (or \defm{classical solution}) if, in addition, 
it is a.e.\ equal to a \emph{Lipschitz-continuous} function.

Weak solutions --- as well as other concepts --- for the \defm{system of conservation laws}
\begin{alignat}{1}
    U_t + f^x(U)_x + f^y(U)_y &= 0 \qquad\text{in $A$,} \myeqlabel{eq:claw}
\end{alignat}
with $f^x,f^y:\Ps\rightarrow\R^m$ smooth, are defined by interpreting
\myeqref{eq:claw} as $2m$ inequalities of the form \myeqref{eq:entropy}, 
with $=$ replaced by $\leq$ or $\geq$ and with $\eta(U):=U^\alpha$, $\psi^x(U):=f^{x\alpha}(U)$, $\psi^y(U):=f^{y\alpha}(U)$
for $\alpha=1,...,m$.

We call $(\eta,\psi^x,\psi^y)$ an \defm{entropy-flux pair} for \myeqref{eq:claw} if
\begin{alignat}{1} 
    \psi^x_U &= \eta_Uf^x_U,\quad \psi^y_U = \eta_Uf^y_U \quad\text{on $\Ps$}. \myeqlabel{eq:eef}
\end{alignat}
A weak solution $U$ of \myeqref{eq:claw} is called \defm{entropy solution} (or \defm{admissible}) if it satisfies 
\myeqref{eq:entropy} for all entropy-flux pairs with convex $\eta$. However, all results in our paper hold
even if we require \myeqref{eq:entropy} only for a \emph{single} entropy-flux pair with uniformly convex $\eta$.

(An alternative to entropy-based admissibility are Lax conditions; see Section 15, especially 
\myeqref{eq:forward-lax} and \myeqref{eq:backward-lax}, for Lax-type conditions suitable for our $L^\infty$ setting.)

\section{Steady and self-similar solutions}

We are interested in \defm{steady} solutions: $U$ is (a version\footnote{i.e.\ almost everywhere equal to} of a function that is) constant in $t$.
Integrating by parts in $t$ eliminates the first term in the integrand in \myeqref{eq:weak-entropy}, after which using compact-in-$t$ support and integrating with respect to $t$ yields the equivalent statement
\begin{alignat}{1}
    - \int_{\R^2} \Phi_x\psi^x(U)+\Phi_y\psi^y(U) d(x,y) &\leq 0 \myeqlabel{eq:weak-entropy-steady}
\end{alignat}
for all nonnegative smooth compactly supported (now taken to be $t$-independent) functions $\Phi:\R^2\rightarrow\R$.

In addition we require $U$ to be \defm{self-similar}:
a.e.\ equal to a function that is constant on each ray $\{s\cdot(x,y):s>0\}$ (for $(x,y)\neq 0$).
While the derivation of the weak form is elementary,
we show it in detail in order to point out some subtleties:
To derive the weak form, first consider nonnegative smooth $\Phi$ whose compact support
is contained in the halfplane $\{x>0\}$.
We may change variables in \myeqref{eq:weak-entropy-steady} to $(x,\xi)$ with $\xi=y/x$:
\begin{alignat}{1}
    0
    &\geq
    - \int_0^\infty \int_{\R} \Big( \Phi_x(x,x\xi) \psi^x\big(U(\xi)\big) 
    + \Phi_y(x,x\xi) \psi^y\big(U(\xi)\big) \Big) x~d\xi~dx.
    \myeqlabel{eq:wec}
\end{alignat}
We take
\begin{alignat}{1}
    \phi(\xi) 
    &:= 
    \int_0^\infty 1\cdot\Phi(x,x\xi) dx  
    = 
    -\int_0^\infty x \big( \Phi_x(x,x\xi) + \xi \Phi_y(x,x\xi) \big) dx 
\notag\end{alignat}
so that
\begin{alignat}{1}
    \phi_\xi(\xi) 
    &= 
    \int_0^\infty x \Phi_y(x,x\xi) dx,
    \notag\\
    \int_0^\infty x \Phi_x(x,x\xi) dx
    &=
    -
    \phi(\xi)
    -
    \xi \phi_\xi(\xi).
    \myeqlabel{eq:phi-Phi}
\end{alignat}
Then \myeqref{eq:wec} is equivalent to
\begin{alignat}{1}
    0 &\geq 
    \int_\R \phi(\xi)\psi^x\big(U(\xi)\big) - \phi_\xi(\xi) \Big( \psi^y\big(U(\xi)\big) - \xi \psi^x\big(U(\xi)\big) \Big) d\xi 
    \myeqlabel{eq:weak-entropy-stss}
\end{alignat}
for every smooth compactly supported nonnegative $\phi:\R\rightarrow\R$, since every $\phi$ arises from \myeqref{eq:phi-Phi} via
\begin{alignat}{1}&
    \Phi(x,x\xi) := \phi(\xi)\eta(x)
\notag\end{alignat}
where $\eta$ is any smooth function with support in $(0,\infty)$ and integral $1$. \myeqref{eq:weak-entropy-stss}
is the weak formulation of
\begin{alignat}{1}
    \big(\psi^y(U)-\xi\psi^x(U)\big)_\xi + \psi^x(U) &\leq 0. \myeqlabel{eq:entropy-stss}
\end{alignat}

By analogous calculations we obtain
\begin{alignat}{1}
    \big(f^y(U)-\xi f^x(U)\big)_\xi + f^x(U) &= 0. \myeqlabel{eq:claw-stss}
\end{alignat}
If $U$ is differentiable at $\xi$, \myeqref{eq:claw-stss} implies
\begin{alignat}{1}
    \big( f^y_U(U) - \xi f^x_U(U) \big) U_\xi &= 0. \myeqlabel{eq:claw-stss-diffed}
\end{alignat}

If we repeat these arguments for $x<0$, there is a single but crucial difference: the coordinate change to
\myeqref{eq:wec} produces an additional ``$-$'' from 
$$dy=|x|d\xi=-x~d\xi.$$ The sign is irrelevant for \myeqref{eq:claw-stss},
but the entropy inequality \myeqref{eq:entropy-stss} changes to
\begin{alignat}{1}
    \big(\psi^y(U)-\xi\psi^x(U)\big)_\xi + \psi^x(U) &\fbox{$\geq$}~0. \myeqlabel{eq:negentropy-stss}
\end{alignat}

\section{Smallness}

We restrict ourselves to the case where $U$ is $L^\infty$-close to a constant background state $\Ub\in\Ps$:
$$ \| U - \Ub \|_{L^\infty} \leq \ueps .$$
A finite number of times in this article, we choose $\ueps>0$ sufficiently small for some purpose.

\subsection{Entropy gradient}

For any entropy-flux pair $(\eta, \psi^x,\psi^y)$ and any $w\in\R^m$,
$$ \hat\eta(U):=\eta(U)+w\cdot U, \quad 
\hat\psi^i(U):=\psi^i(U)+w\cdot f^i(U) \qquad(i=x,y)
$$
defines another entropy-flux pair $(\hat\eta,\hat\psi^x,\hat\psi^y)$ since
$$ \hat \psi^i_U = \psi^i_U+w\cdot f^i_U = (\eta_U+w^T) f^i_U = \hat\eta_U f^i_U \qquad(i=x,y). $$
$\hat\eta_{UU}=\eta_{UU}$, so convexity is not affected.
By adding $w^T$ times \myeqref{eq:claw} to \myeqref{eq:entropy} (which is a linear operation,
hence compatible with weak formulation) we obtain
\begin{alignat}{1}
    \hat\eta(U)_t + \hat\psi^x(U)_x + \hat\psi^y(U)_y &\leq 0
\notag\end{alignat}
which is equivalent to \myeqref{eq:entropy} since we can reverse the operation using $-w^T$.
Hence we may assume, without loss of generality, that 
\begin{alignat}{1}
    \eta_U(\Ub) &= 0  \myeqlabel{eq:egrad}
\end{alignat}
and do so from now on.

\section{Eigenvalues}

From \myeqref{eq:claw-stss-diffed}, we see $U_\xi=0$ is implied if the matrix
$$ f^y_U(U)-\xi f^x_U(U) $$
is not singular. 
This suggests, as we show later, that $U$ is constant in sectors where the matrix 
\begin{alignat}{1}&
    f^y_U(\Ub)-\xi f^x_U(\Ub)  \myeqlabel{eq:inhom}
\end{alignat}
is far from singular, so that the interesting behaviour is concentrated near $\xi$ that satisfy
\begin{alignat}{1}
    0 &= \det\big(f^y_U(\Ub)-\xi f^x_U(\Ub)\big) := p(\xi) \myeqlabel{eq:hom}
\end{alignat}
for our constant background state $\Ub$.  The polynomial $p$ has up to $m$ real roots. 

Instead of focusing on one choice of coordinates, consider 
\begin{alignat}{1}
    P(x:y) &= \det\big(\vec x\times\vec f_U(\Ub)\big) 
    \myeqlabel{eq:bigP}
\end{alignat}
(with $\vec x=(x,y)$, $\vec f=(f^x,f^y)$ and $(a_1,b_1)\times(a_2,b_2):=a_1b_2-a_2b_1$) 
where we regard $(x:y)\in\RP^1$ as homogeneous coordinates; $P$ is a homogeneous polynomial of degree $\leq m$.
``$\times$'' is invariant under rotation, so that a coordinate change from $(x,y)$ to $(x',y')=R(x,y)$, $R$ any rotation matrix,
changes each root of $P$ from $(x:y)$ to $(x':y')$. Each root $\xi$ of $p$ corresponds to a root $(1:\xi)$ of $P$.

Since $p$ has $\leq m$ roots, we can find some $\xi$ which is not a root. The line $(1:\xi)$ associated to $\xi$ is,
by rotating coordinates, aligned with $(0:1)$. Then 
\begin{alignat}{1}
    P(0:1) &\neq 0 .\myeqlabel{eq:novert}
\end{alignat}
We assume from now on, without loss of generality, that this change has been made.

\section{Change of variables}

\subsection{Change to $V$}

\myeqref{eq:novert} also implies that $f^x_U$ is \emph{regular}. 
Therefore $f^x$ is a diffeomorphism if we choose $\ueps>0$ sufficiently small. 
Since we are using $U_t=0$, it is not important to work with conserved quantities and we may change to
$$ V:= f^x(U) ,\qquad \Vb:=f^x(\Ub),$$
and set 
$$ f(V):=f^y(U(V)) .$$
We let $\Ps$ be the open set of possible values for $V$ from now on and abbreviate
$$ \Peps := \Big\{V\in\Ps~\Big|~|V-\Vb|\leq\ueps\Big\}. $$
\myeqref{eq:claw-stss} 
becomes
\begin{alignat}{1}&
    (f(V)-\xi V)_\xi + V = 0. \myeqlabel{eq:claw-V}
\end{alignat}
At points of differentiability $\xi$ of $V$ we have
\begin{alignat}{1}
    \big(f_V(V(\xi))-\xi I\big)V_\xi(\xi) &= 0 .\myeqlabel{eq:claw-V-diffed}
\end{alignat}
These are the \emph{same} equations satisfied by a weak solution $V$
of a 1-d conservation law 
$$ V_t + f(V)_z = 0 $$
that is self-similar, i.e.\ 
$$V(z,t)=V(\xi),$$
if we identify $\xi=z/t$.
Hence our $x$ is a time-like variable while $y$ is space-like.
We could, for example, solve an initial-value problem by imposing data at a fixed $x$.
However, there is no well-posedness without an \emph{entropy inequality} which
is what identifies the forward and backward directions of time in physics.

\subsection{Entropy-Flux pairs} 

For any entropy-flux pair $(\eta, \psi^x, \psi^y)$, define

\begin{alignat}{1}
e(V) := \psi^x(U(V)), \qquad q(V) := \psi^y(U(V)).  \nonumber
\end{alignat}
Then
\begin{alignat}{1}
e_V = \psi^x_U U_V = \eta_U f^x_U U_V = \eta_U V_U U_V = \eta_U. \nonumber
\end{alignat}
Therefore,
\begin{alignat}{1}
e_{VV} = \eta_{UU} U_V = \eta_{UU} (f^x_U)^{-1}. \nonumber
\end{alignat}
We have
\begin{alignat}{1}
q_V = \psi^y_U U_V = \eta_U f^y_U U_V = e_V f_V, \nonumber
\end{alignat}
since
\begin{alignat}{1}
f_V = f^y_U U_V. \nonumber
\end{alignat}
Therefore, $(e, q)$ is an entropy-flux pair for \myeqref{eq:claw-V}.  The entropy inequality \myeqref{eq:entropy-stss} for $x>0$ becomes
\begin{alignat}{1}
    \big(q(V)-\xi e(V)\big)_\xi + e(V) &\leq 0,
    \myeqlabel{eq:entropy-V}
\end{alignat}
whereas
\begin{alignat}{1}
    \big(q(V)-\xi e(V) \big)_\xi + e(V) &\fbox{$\geq$}~ 0
    \myeqlabel{eq:negentropy-V}
\end{alignat}
for $x < 0$.

\subsection{Convex Entropy}

\begin{lemma}
If $f^x(U)$ is regular, then $e_{VV}r^{\alpha}r^{\alpha} \neq 0$ for all $\alpha = 1,..m$.  If $f^x(U)$ has only positive (negative) eigenvalues, then $e$ is uniformly convex (concave).
\end{lemma}

\begin{proof}
We shall use Proposition 6.1 from \cite{serre-matrices-secondedition}.  
It states that if $H$ is symmetric positive definite, and $K$ is symmetric, then $HK$ is diagonalizable with real eigenvalues.  Moreover, the number of positive (negative) eigenvalues of $K$ equals the number of positive (negative) eigenvalues of $HK$.  First, we write
\begin{alignat}{1} 
(f^x_U)^{-1} = (\eta_{UU})^{-1} e_{VV}. \notag
\end{alignat}
$(\eta_{UU})^{-1}$ is symmetric positive definite, and $e_{VV}$ is symmetric.  Then, applying the proposition, since $(f^x_U)^{-1}$ is nondegenerate, all eigenvalues of $e_{VV}$ are nonzero.  Moreover, if all the eigenvalues of $(f^x_U)^{-1}$ are positive (negative), then $e_{VV}$ is positive (negative) definite, since a symmetric matrix is positive (negative) definite if and only if its eigenvalues are all positive (negative).  All that is left is to show that $e_{VV}r^{\alpha}r^{\alpha} \neq 0$  in the case of eigenvalues of mixed signs. \\

As in \cite[Lemma 4.3.3]{serre-1}, we consider
\begin{alignat}{1}
q_V = e_V f_V. \notag
\end{alignat}
Then,
\begin{alignat}{1}
q_{VV} = e_{VV} f_V + e_V f_{VV}.  \notag
\end{alignat}
Therefore,
\begin{alignat}{1}
e_{VV}f_V = q_{VV} - e_V f_{VV}. \notag
\end{alignat}
The first term on the right side is symmetric, and the second term on the right is a linear combination of symmetric matrices, and is thus symmetric.  Therefore, the left side is also symmetric and thus defines a symmetric bilinear form.  Then
\begin{alignat}{1}
 e_{VV} (f_V r^\alpha) r^\beta &= e_{VV} (f_V r^\beta) r^\alpha  \notag \\
 \lambda^\alpha e_{VV} r^\alpha r^\beta &= \lambda^\beta e_{VV} r^\beta r^\alpha \notag \\
 (\lambda_\alpha-\lambda_\beta) e_{VV} r^\alpha r^\beta &= 0. \notag
\end{alignat}
Therefore, for $\beta \neq \alpha$, $e_{VV}r^{\alpha}r^{\beta} = 0$ by strict hyperbolicity.  Suppose that
\begin{alignat}{1}
e_{VV}r^{\alpha} r^{\alpha} = 0. \notag
\end{alignat}
By bilinearity, this would imply that 
\begin{alignat}{1}
e_{VV}r^{\alpha}s = 0 \notag
\end{alignat}
for all $s \in \R^m$.  Therefore $e_{VV}r^{\alpha}$ must be the zero vector, but this contradicts the fact that $e_{VV}$ has all eigenvalues nonzero.  Therefore, for each $\alpha$,
\begin{alignat}{1}
e_{VV}r^{\alpha} r^{\alpha} \neq 0. \notag
\end{alignat}
\end{proof}

\section{Versions}

\label{section:vers}

Consider \myeqref{eq:claw-V}. $f(V(\xi))-\xi V(\xi)$ has a distributional derivative $-V\in L^\infty$,
so 
there is a $C\in\R^m$ so that 
\begin{alignat}{1}
    f(V(\xi))-\xi V(\xi) &= C - \int_0^\xi V(\eta) d\eta \qquad\text{for a.e.\ $\xi\in\R$.}
    \myeqlabel{eq:fint}
\end{alignat}
Analogously, \myeqref{eq:entropy-V} yields a $C'\in\R$ with
\begin{alignat}{1}
    q(V(\xi))-\xi e(V(\xi)) &\leq C' - \int_0^\xi e(V(\eta)) d\eta \qquad\text{for a.e.\ $\xi\in\R$.}
    \myeqlabel{eq:qint}
\end{alignat}
Since the left-hand sides are continuous functions of $V(\xi)$ and the right-hand sides continuous functions of $\xi$,
Lemma \myref{lem:version-adapt} from the appendix applied to \myeqref{eq:fint} (with $=$ split into $\leq,\geq$) and \myeqref{eq:qint}
yields a version (that is, an element of the $L^{\infty}$ equivalence class containing $V$, which we will continue to refer to as $V$) of $V$ that (a) has values in $\Peps$ \emph{everywhere}, and (b) so that for \emph{all} $\xi_1,\xi_2\in\R$
\begin{alignat}{1}
    -\int_{\xi_1}^{\xi_2} V(\eta) d\eta 
    &=
    \Big( f\big(V(\xi_2)\big) - \xi_2 V(\xi_2) \Big) - \Big( f\big(V(\xi_1)\big) - \xi_1 V(\xi_1) \Big)
    \qquad\text{and} \myeqlabel{eq:fintfull}\\
    -\int_{\xi_1}^{\xi_2} e\big(V(\eta)\big) d\eta 
    &\geq
    \Big( q\big(V(\xi_2)\big) - \xi_2 e\big(V(\xi_2)\big) \Big) - \Big( q\big(V(\xi_1)\big) - \xi_1 e\big(V(\xi_1)\big) \Big). 
\notag\end{alignat}
We abbreviate
\begin{alignat}{1}
    \hat A(V_0+\Delta V,V_0) &:= \int_0^1 f_V(V_0+s\Delta V) ds \myeqlabel{eq:hatAdef}
\end{alignat}
and obtain
\begin{alignat}{1}
    f(V_0+\Delta V)-f(V_0) &= \hat A(V_0+\Delta V,V_0) \Delta V,
\notag\end{alignat}
so
\begin{alignat}{1}
    \int_{\xi_1}^{\xi_2} V(\xi_2)-V(\eta) d\eta 
    &=
    \Big( \hat A\big(V(\xi_2),V(\xi_1)\big)-\xi_1 I \Big) \big( V(\xi_2)-V(\xi_1) \big).
    \myeqlabel{eq:Adiff}
\end{alignat}

\section{Strict hyperbolicity}

For the remainder of the paper we focus on the case of strict hyperbolicity. 
Many results would hold for weaker notions of hyperbolicity,
but we prefer to keep the presentation simple.
By \defm{strict hyperbolicity} we mean that
$P$ in \myeqref{eq:bigP} has exactly $m$ \emph{real} roots $(x:y)$ which are necessarily distinct.
That means 
$$ \det(f_V(\Vb)-\xi I) = 0 $$
has $m$ distinct real roots $\xi$.

$\hat A(\Vb,\Vb) = f_V(\Vb)$, so by smoothness of $\hat A$
we can take $\ueps>0$ so small that for $V^\pm\in\Peps$ there are $m$ real eigenvalues 
$\hat\lambda^\alpha(V^\pm)$ ($\alpha=1,...,m$) of $\hat A(V^\pm)$ which are smooth functions of $V^\pm$ and satisfy
\begin{alignat}{1}&
    \hat\lambda^\alpha(V^\pm) < \hat\lambda^{\alpha+1}(\tilde V^\pm) \qquad\forall V^\pm,\tilde V^\pm\in\Peps,~\alpha\in\{1,...,m-1\}.
    \myeqlabel{eq:stricthyperhat}
\end{alignat}
($\Peps$ is compact, so the separation is uniform, by continuity of $\hat\lambda^\alpha$.
The $\hat\lambda^\alpha$ must remain distinct and \emph{real} because their $m$ real parts are continuous functions of $V^\pm$,
hence remain distinct for $\ueps>0$ sufficiently small, 
so since $\hat A(V^\pm)$ is real it cannot have non-real eigenvalues 
which come in conjugate pairs which would yield two equal real parts.)

For $\alpha=1,...,m$ we choose a unit-length right eigenvector $\hat r^\alpha(V^\pm)$ 
of $\hat A(V^\pm)$ for eigenvalue $\hat\lambda^\alpha(V^\pm)$.
$\hat r^\alpha(V^\pm)$ is also a smooth function of $V^\pm$. We choose left eigenvectors $\hat l^\alpha(V^\pm)$ that satisfy 
$$ \hat l^\alpha \hat r^\beta = \delta_{\alpha\beta} \qquad (\alpha,\beta=1,...,m),$$
which implies they are smooth as well.

Abbreviate
$$ A(V) := \hat A(V,V) = f_V(V), \quad \lambda^\alpha(V):=\hat\lambda^\alpha(V,V),\quad r^\alpha(V):=\hat r^\alpha(V,V),\quad l^\alpha(V):=\hat l^\alpha(V,V).$$

\section{Left and right sequences}

In this article we do not assume $V\in\BV$, so $V$ need not have well-defined left or right limits at any point $\xi$.
Instead we consider pairs of sequences $(\tilde\xi^-_k),(\tilde\xi^+_k)$, both converging to $\xi$, with $\tilde\xi^-_k<\tilde\xi^+_k$
(we do not require $\tilde\xi^-_k<\xi<\tilde\xi^+_k$ yet).
Since $V$ has values in the compact set $\Peps$, there are subsequences $(\xi^+_k)$ of $(\tilde\xi^+_k)$
and $(\xi^-_k)$ of $(\tilde\xi^-_k)$ so that 
\begin{alignat}{1}
    V(\xi^+_k)\rightarrow V^+ \quad,\quad V(\xi^-_k)\rightarrow V^- . \myeqlabel{eq:Vleftright}
\end{alignat}
In such a context we write
$$ [g(V)] := g(V^+)-g(V^-) $$
for any function $g$ (assuming there is no ambiguity as to which sequences are meant).

Let 
$$ J(g(V);\xi) := \sup \big|[g(V)]\big| $$
where the $\sup$ is over all sequences $(\xi^\pm_k)$ with the properties above.
Then $J(g(V);\xi)=0$ if and only if $g\circ V$ is continuous at $\xi$.

By \myeqref{eq:Adiff},
\begin{alignat}{1}
    \Big(\hat A\big(V(\xi_k^+),V(\xi_k^-)\big)-\xi_k^- I \Big) \big(V(\xi^+_k)-V(\xi^-_k)\big)
    &= 
    \int_{\xi_k^-}^{\xi_k^+} V(\xi_k^+) - V(\eta) d\eta. 
\notag\end{alignat}
The limit as $k\rightarrow\infty$ is
\begin{alignat}{1}&
    \big(\hat A(V^\pm)-\xi I\big)[V] = 0. \myeqlabel{eq:Ajump}
\end{alignat}
Hence for some $\alpha\in\{1,...,m\}$
\begin{alignat}{1}    
    [V]\parallel \hat r^\alpha(V^\pm) \quad\text{and}\quad \xi = \hat\lambda^\alpha(V^\pm)
    \myeqlabel{eq:jump-spec}
\end{alignat}
(that is, $[V]$ is a scalar multiple of $\hat r^\alpha(V^\pm)$).
\myeqref{eq:Ajump} is equivalent to
\begin{alignat}{1}
    [f(V)] - \xi[V] &= 0 \myeqlabel{eq:rh}
\end{alignat}
which is the usual \defm{Rankine-Hugoniot condition}.
Hence we may use any standard result that does not require continuity on each side of $\xi$.

\section{General case}

In this section we collect results that do not require any assumption
(such as strict hyperbolicity, admissibility, genuine nonlinearity, ...).

\begin{theorem}
    \mylabel{th:Uconst}
    Suppose $V$ is continuous on an interval $I=\boi{\xi_1}{\xi_2}$ and that
    $\xi$ is not an eigenvalue of $A(V(\xi))$
    for any $\xi\in I$.
    Then $V$ is constant on $I$.
\end{theorem}
\begin{proof}
    Fix some $\xi \in I$.
    We claim that $V$ must be 
    Lipschitz at $\xi$.  
    Suppose not.  
    Then we can choose a sequence $\left\{h_n \right\} \rightarrow 0$ (with $h_n\neq 0$) such that 
    $$ 0 < \Big|\frac{V(\xi+h_n)-V(\xi)}{h_n}\Big| \nearrow\infty. $$
    Divide both sides of \myeqref{eq:Adiff} by $|V(\xi+h_n)-V(\xi)|$ to obtain
    \begin{alignat}{1}&
        \Big(\hat A\big(V(\xi+h_n),V(\xi)\big)-\xi I\Big) \frac{V(\xi+h_n)-V(\xi)}{|V(\xi+h_n)-V(\xi)|} 
        \notag\\&= 
        \frac{1}{|V(\xi+h_n)-V(\xi)|} \subeq{\int_{\xi}^{\xi+h_n} V(\xi+h_n)-V(\eta) d\eta }{=O(h_n)}
        = o(1)\quad\text{as $n\rightarrow\infty$}
        \myeqlabel{eq:AUU}
    \end{alignat}
    ($O(h_n)$ since $V$ is bounded). By assumption, $A(V(\xi))-\xi I$ is regular, 
    so for $h$ sufficiently small $\hat A\big(V(\xi+h),V(\xi)\big)-\xi I$ will be uniformly regular. That is,
    \[ \exists \delta>0~\forall v\in\R^m:\Big|\Big(\hat A\big(V(\xi+h),V(\xi)\big)-\xi I\Big) v\Big| \geq \delta|v| \]
    Taking $n \rightarrow \infty$, the left hand side of \myeqref{eq:AUU} stays bounded away from zero, 
    while the right hand side goes to zero, 
    leading to a contradiction.  

    Therefore, $V$ must be Lipschitz on $I$.
    Assuming $\xi$ is a point of differentiability of $V$, we obtain
    $$ \big(A(V(\xi))-\xi V(\xi)\big) V_\xi = 0.$$
    However, as we assumed the matrix was regular on $I$, it follows that $V_\xi=0$ a.e.\ on $I$.
    A Lipschitz function is the integral of its derivative, so $V$ is constant on $I$.
\end{proof}

\begin{theorem}
    \mylabel{th:Uconst2}%
    Consider an interval $I=\boi{\xi_1}{\xi_2}$.
    There is a $\delta_s=\delta_s(\ueps)>0$, with
    $$ \delta_s\downarrow 0\quad\text{as}\quad\ueps\downarrow 0, $$
    so that 
    \begin{alignat}{1}
        \forall\alpha\in\{1,...,m\}\forall x\in I: |\lambda^\alpha(V(\xi))-\xi| > \delta_s
        \myeqlabel{eq:lamdel}
    \end{alignat}
    implies $V$ is constant on $I$.
    [Here we do not require continuity of $V$, but a stronger bound on the spectrum.]
\end{theorem}
\begin{proof}
    Define
    $$ \delta_s := \sup_{V,V^\pm \in \Peps} |\lambda^\alpha(V) - \hat\lambda^\alpha(V^\pm)| $$
    and assume \myeqref{eq:lamdel} holds.
    The right-hand side converges to zero as $\ueps\searrow 0$ since $\lambda^\alpha,\hat\lambda^\alpha$ are smooth
    and coincide for $V=V^+=V^-$.

    Assume $V$ is discontinuous at $\xi\in I$. 
    Then we may choose $(\xi^+_k),(\xi^-_k)\rightarrow\xi$ with $V(\xi^\pm_k)\rightarrow V^\pm$ 
    and $[V]\neq 0$
    and obtain, by \myeqref{eq:jump-spec}, that 
    $$ \xi = \hat\lambda^\alpha(V^\pm) .$$ 
    But then 
    $$ |\xi - \lambda^\alpha(V(\xi))| \leq \delta_s, $$
    which contradicts \myeqref{eq:lamdel}.
    
    Hence $V$ is continuous on $I$; Theorem \myref{th:Uconst} yields the conclusion.
\end{proof}

\section{Vertical axis neighbourhood}

As explained in the context of \myeqref{eq:novert}, we may choose some $\xi\in\R$ that is not a root of
$p$ (see \myeqref{eq:inhom}) in the present coordinates and rotate coordinates so that $(1:\xi)$ is aligned with $(0:1)$
and therefore $(0:1)$ with $(-1:\xi)$. Then $-\xi$, by \myeqref{eq:novert}, is not a root of $p$ in new coordinates, so
Theorem \myref{th:Uconst2} shows (if $\ueps>0$ is sufficiently small) that 
$U(\eta)$ must be constant for $\eta$ in a neighbourhood of $-\xi$.
Rotating back to old coordinates it is constant --- and therefore a weak solution --- in
sufficiently narrow open convex cones containing the positive and negative vertical axis.  Therefore, we lost no generality 
by considering test functions supported away from the $y$-axis while deriving the weak form.

\section{Sectors}

By Theorem \myref{th:Uconst2}, we can choose $\ueps$ so small that there are intervals
$$ I^\alpha := \boi{\lambda^\alpha(\Vb)-\delta^\alpha}{\lambda^\alpha(\Vb)+\delta^\alpha} \qquad (\alpha=1,...,m) $$
for $\delta^\alpha>0$ so that $V$ is constant outside $\bigcup_{i=1}^mI^\alpha$. We may choose $\delta^\alpha\downarrow 0$ as $\ueps\downarrow 0$.
By \defm{forward sector} (see Figure \myref{fig:test})
we mean $\xi\in I^\alpha$ with $x>0$, whereas \defm{backward sector} refers to $x<0$.

\section{Genuine nonlinearity}

\begin{definition}
    We say $I^\alpha$ is \defm{genuinely nonlinear} if 
    \begin{alignat}{1}
        \forall V\in\Peps : \lambda^\alpha_V(V)r^\alpha(V) &> 0. \myeqlabel{eq:gennon}
    \end{alignat}
    (if $<0$ we may without loss of generality flip the sign of $r^\alpha(V),\hat r^\alpha(V^\pm)$ (which remain unit-length)
    and $l^\alpha(V),\hat l^\alpha(V^\pm)$).
    We say $I^\alpha$ is \defm{linearly degenerate} if 
    \begin{alignat}{1}
        \forall V\in\Peps : \lambda^\alpha_V(V)r^\alpha(V) &= 0. \myeqlabel{eq:lindeg}
    \end{alignat}
\end{definition}

\section{Simple waves}

\subsection{Simple wave curves}

Let $s\mapsto R^\alpha(V^-,s)$ solve
$$ R^\alpha(V^-,0)=V^-, \qquad R^\alpha_s(V^-,s) = r^\alpha(R^\alpha(V^-,s)). $$
$R^\alpha$ defines the \defm{$\alpha$-simple wave curve}. For each $V^-$ we take the interval for $s$ maximal 
so that $R^\alpha(V^-,s)\in\Peps$.

\subsection{Wave fans}

If $I^\alpha$ is genuinely nonlinear, then 
$$ \lambda^\alpha(R^\alpha(V^-,s))_s
= \lambda^\alpha_V(R^\alpha(V^-,s))r^\alpha(R^\alpha(V^-,s))
>0 ,$$
so 
\begin{alignat}{1}&
    s \mapsto \lambda^\alpha(R^\alpha(V^-,s)) \quad \text{is \emph{strictly increasing}.} \myeqlabel{eq:lam-xi}
\end{alignat}
Let $\xi\mapsto s(\xi)$ be its inverse map. By setting
$$ W(\xi) := R^\alpha(V^-,s(\xi)) \qquad\text{for $\xi\geq\lambda^\alpha(V^-)$,} $$
we obtain a strong solution of \myeqref{eq:claw-V-diffed} since 
$$ 
\Big(A(W(\xi)) - \xi I\Big) W_\xi
= \subeq{\Big(A\big(W(\xi)\big) - \lambda^\alpha\big(W(\xi)\big) I\Big) r^\alpha\big(W(\xi)\big)}{=0} s_\xi(\xi)
= 0
.$$
If we interpret $V^-$ as the value of $W$ at the \emph{smallest} $\xi$, 
then only the $s\geq 0$ part of $R^{\alpha}$, denoted $R^{\alpha+}$, is relevant.

\section{Discontinuities}

We recall some standard results we need later,
to show that they do not depend on having a smooth neighbourhood on each side of a discontinuity.

\subsection{Shock curves}

Consider sequences $(\xi^+_k)$ and $(\xi^-_k)$ converging to $\xi$, with $\xi^-_k<\xi^+_k$ for all $k$,
so that $V(\xi^\pm_k)\rightarrow V^\pm$. This is the setting of \myeqref{eq:jump-spec} which implies
$[V^\pm]$ is a right eigenvector of $\hat A(V^\pm)$ and $\xi$ the corresponding eigenvalue.
So there is an $\alpha\in\{1,...,m\}$ with 
$$ h(V^+,s) := V^+ - V^- - s\hat r^\alpha(V^\pm) = 0. $$
$h$ is smooth, $h(V^-,0)=0$ and
$$ \frac{\partial h}{\partial V^+}(V^-,0) = I, $$
so the implicit function theorem yields, after taking $\ueps>0$ sufficiently small,
existence of a smooth bijective map $s\mapsto S^\alpha(V^-,s)$ with 
$$ S^\alpha(V^-,0)=V^-, \qquad S^\alpha(V^-,s) - V^- - s\hat r^\alpha(V^\pm) = 0. $$
For each $V^-$ we take the interval for $s$ maximal so that $S^\alpha(V^-,s)\in\Peps$.
$S^\alpha(V^-,\cdot)$ defines the \defm{$\alpha$-shock curve} of $V^-$.
It contains $V^-$ (via $s=0$) and has tangent $r^\alpha(V^-)$ there.

We take $\ueps>0$ so small that for each $\alpha$ only $V^+=S^\alpha(V^-,s)$ are solutions of \myeqref{eq:rh}.

\subsection{Contact curves}

Assume $I^\alpha$ is linearly degenerate. 
Then 
$$ \lambda^\alpha(R^\alpha(V^-,s))_s
= \lambda^\alpha_V(R^\alpha(V^-,s))r^\alpha(R^\alpha(V^-,s))
\topref{eq:lindeg}{=} 0. $$
Hence 
\begin{alignat}{1}&
    s \mapsto \lambda^\alpha(R^\alpha(V^-,s)) \quad \text{is \emph{constant}.} \myeqlabel{eq:lam-const}
\end{alignat}

Now consider
$$ F(s) := f(R^\alpha(V^-,s)) - f(V^-) - \xi (R^\alpha(V^-,s) - V^-) .$$
Then $F(0)=0$, and
$$ F_s(s) = A(R^\alpha(V^-,s)) r^\alpha(R^\alpha(V^-,s)) - \xi r^\alpha(R^\alpha(V^-,s)) .$$
This is zero if we set $\xi=\lambda^\alpha(R^\alpha(V^-,s))$ which is possible since the latter is constant.
Hence the Rankine-Hugoniot condition \myeqref{eq:rh} is satisfied.

Since $R^\alpha$ is maximal in $\Peps$, since $S^\alpha$ is maximal as well and contains
the only points in $\Peps$ satisfying \myeqref{eq:rh}, and since both are simple smooth curves,
they are identical.

Hence, at $\xi$ where an $\alpha$-contact --- $[V]$ a right eigenvector for $\hat\lambda^\alpha(V^\pm)$ --- 
occurs, we have
\begin{alignat}{1}
    \lambda^\alpha(V^-) = \xi = \hat\lambda^\alpha(V^\pm) = \lambda^\alpha(V^+). \myeqlabel{eq:lam-xi-const}
\end{alignat}

\subsection{Admissible shock curve}

Now assume $I^\alpha$ is genuinely nonlinear.
Assume $V$ is admissible. Consider a \emph{forward} sector first. The entropy inequality
\begin{alignat}{1}&
    [q(V)] - \xi [e(V)] \leq 0 \myeqlabel{eq:rh-entropy}
\end{alignat}
can be derived from \myeqref{eq:entropy-V} in the same way as \myeqref{eq:rh} from \myeqref{eq:claw-V}.

By \myeqref{eq:jump-spec}, a jump from $V^-$ to $V^+$ must be located at $\xi=\hat\lambda^\alpha(V^+,V^-)$, and
$$ \hat\lambda^\alpha(V^-,V^-) = \lambda^\alpha(V^-), $$
so
\begin{alignat}{1}&
    \partial_1\hat\lambda^\alpha(V^-,V^-) + \partial_2\hat\lambda^\alpha(V^-,V^-) = \lambda^\alpha_V(V^-)
    \myeqlabel{eq:l12a}
\end{alignat}
Moreover,
$$ \hat\lambda^\alpha(V^-,V^+) = \hat\lambda^\alpha(V^+,V^-) $$
since
$$ \hat A(V^+,V^-) \topref{eq:hatAdef}{=} \int_0^1 f_V((1-s)V^-+sV^+)ds = \int_0^1 f_V(rV^-+(1-r)V^+)dr = \hat A(V^-,V^+).$$
Therefore
\begin{alignat}{1}&
    \partial_1\hat\lambda^\alpha(V^-,V^-) = \partial_2\hat\lambda^\alpha(V^-,V^-) .
    \myeqlabel{eq:l12b}
\end{alignat}
Combining \myeqref{eq:l12a} and \myeqref{eq:l12b} we have
$$ \partial_1\hat\lambda^\alpha(V^-,V^-) 
= \frac{\lambda^\alpha_V(V^-)}{2} 
= \partial_2\hat\lambda^\alpha(V^-,V^-) $$
Thus
\begin{alignat}{1}
    \hat\lambda^\alpha(S^\alpha(V^-,s),V^-)_s
    &\overset{s=0}{=}
    \partial_1\hat\lambda^\alpha(V^-,V^-)S^\alpha_s(V^-,0)
    =
    \frac12\lambda^\alpha_V(V^-)r^\alpha(V^-) > 0.
\notag\end{alignat}
Hence for $\ueps>0$ sufficiently small 
\begin{alignat}{1}
    s\mapsto\hat\lambda^\alpha(S^\alpha(V^-,s),V^-) \quad\text{is strictly increasing.}
    \myeqlabel{eq:lambda-shock}
\end{alignat}
We may reparametrize the $\alpha$-shock curve of $V^-$ to be $\hat\lambda^\alpha=\xi\mapsto W(\xi)$.

Abbreviate $\xb:=\lambda^\alpha(V^-)$.
To avoid clutter we change coordinates so that $e(V^-)=0$, $q(V^-)=0$, $f(V^-)=0$, $W(\xb)=V^-=0$
(which is acceptable since adding constants to $V,f,e$ or $q$ has no effect in \myeqref{eq:claw-V} and \myeqref{eq:entropy-V}).
\myeqref{eq:rh} becomes
$$ 0 = f(W(\xi)) - \xi W(\xi) ,$$
with derivative
\begin{alignat}{1}&
   0 = (f_V(W)-\xi I) W_\xi - W.
   \myeqlabel{eq:fder}
\end{alignat}
\myeqref{eq:rh-entropy} is equivalent to $E(\xi)\leq 0$ for
$$ E(\xi) := q(W(\xi)) - \xi e(W(\xi)) .$$
We analyze the situation near $\xi=\xb$.  Since
$$ E(\xb) = q(W(\xb)) - \xb e(W(\xb)) = q(0) - \xb e(0) = 0,$$
so we need to consider the first derivative, given by
$$ E_\xi = (q_V-\xi e_V)W_\xi - e = e_V(f_V-\xi I)W_\xi - e \topref{eq:fder}{=} e_VW - e .$$
Then
$$ E_\xi(\xb) = e_V(W(\xb))\subeq{W(\xb)}{=0} - \subeq{e(W(\xb))}{=0} = 0. $$
Hence we need to consider the second derivative as well:
$$ E_{\xi\xi} = e_{VV} W_\xi W + e_V W_\xi - e_V W_\xi = e_{VV} W_\xi W .$$
Then
$$ E_{\xi\xi}(\xb) = e_{VV}(W(\xb)) W_\xi(\xb) \subeq{W(\xb)}{=0} = 0.$$
The third derivative finally yields a result:
$$ E_{\xi\xi\xi} = (e_{VVV} W_\xi W_\xi + e_{VV} W_{\xi\xi}) W + e_{VV} W_\xi W_\xi $$
so
$$ E_{\xi\xi\xi}(\xb) \overset{W(\xb)=0}{=} 
e_{VV}(W(\xb)) W_\xi(\xb) W_\xi(\xb) \neq 0, $$
because $W(\xi_0) = V^-$ and $W_\xi(\xb) = r^\alpha(V^-)$ by definition of the shock curve.

Hence, if $e_{VV}r^{\alpha}r^{\alpha} > 0$, for $\ueps>0$ sufficiently small, 
$$ E\leq 0 \quad\Leftrightarrow\quad \xi\leq\xb. $$
Therefore, in this case, only the $s\leq 0$ part of the shock curve (corresponding to $\xi\leq\xb$ 
due to $\hat\lambda^\alpha$ strictly increasing) yields admissible shocks. We call this part $S^{\alpha-}$.

If $e_{VV}r^{\alpha}r^{\alpha} < 0$, then the $\xi \geq \xb$ part is relevant. 

To this end, if 
$$e_{VV}r^\alpha r^\alpha > 0,$$ 
we define the \textit{``forward sector''} to have $x>0$, and the \textit{``backward sector''} to have $x<0$.  Since everything is smooth and this quantity can never be zero, it must be positive for all $V^- \in \Peps$ if it is positive anywhere (vice versa for negative).

Conversely, if 
$$e_{VV}r^{\alpha}r^{\alpha} < 0,$$
the \textit{``forward sector''} has $x<0$ and the \textit{``backward sector''} has $x>0$.  
 
We can consider the same setting but for an $x<0$ sector: an analogous argument, starting with the opposite entropy inequality (21).

Moreover, (42) shows that admissible shocks in \textit{forward} sectors satisfy the \defm{Lax condition}
\begin{alignat}{1}
    \lambda(V^-) > \xi > \lambda(V^+).
\notag\end{alignat}
More precisely the following \defm{uniform} Lax condition holds for \textit{forward} sector shocks:
there is a constant $\delta_L>0$ so that
\begin{alignat}{1}
\ignoretag{44}    \lambda(V^-) - \delta_L\big|[V]\big| \geq \xi \geq \lambda(V^+) + \delta_L\big|[V]\big|. \myeqlabel{eq:forward-lax}
\end{alignat}
Finally we consider the same setting but for a \textit{backward} sector: 
\begin{alignat}{1}
\ignoretag{45}    \lambda(V^-) + \delta_L\big|[V]\big| \leq \xi \leq \lambda(V^+) - \delta_L\big|[V]\big|. \myeqlabel{eq:backward-lax}
\end{alignat}

If the background state for Euler flow is supersonic horizontal velocity to the right, then all forward sectors are $x>0$.  If, however, the background state is supersonic horizontal velocity to the left, then all forward sectors are $x<0$.  For 1-dimensional conservation laws with convex entropy, all forward sectors are $t>0$.  Finally, it is not hard to construct examples that satisfy all assumptions but have $f^x_U$ possessing eigenvalues with different signs, so in certain cases there can be forward sectors for some eigenvalues contained in $x>0$ and others in $x<0$.

\section{Linearly degenerate sectors}

We consider linearly degenerate $I^\alpha$ and allow both $x>0$ and $x<0$.

\begin{lemma}
    \mylabel{lem:contver}%
    (a) $\lambda^\alpha\circ V$ is continuous.
    (b) If $\xi\neq\lambda^\alpha(V(\xi))$ on an open set $A\subset I^\alpha$,
    then $V$ is constant on $A$.
\end{lemma}
\begin{proof}
    Assume $\lambda^\alpha\circ V$ and therefore $V$ are discontinuous at $\xb\in I^\alpha$. Then we can choose $(\xi^\pm_k)\rightarrow\xb$ 
    with $V(\xi^\pm_k)\rightarrow V^\pm$ so that
    $$ [\lambda^\alpha(V)] \neq 0 .$$
    However, since $I^\alpha$ is linearly degenerate and $V^+$ is on the $\alpha$-simple wave curve of $V^-$, 
    \myeqref{eq:lam-xi-const} shows
    $$ \lambda^\alpha(V^+) = \lambda^\alpha(V^-) = \xb, $$
    contradicting the assumption that $\lambda^\alpha \circ V$ is discontinuous at $\xi_0$. This shows (a).
    Theorem \myref{th:Uconst} yields (b).
\end{proof}

\begin{lemma}
	\mylabel{lem:subseq}%
For any subset $E \subset I^{\alpha}$, and for almost every $\xi_0 \in E$, there exists $D \subset E$ containing $\xi_0$ such that
\[(V_{|D})'(\xi_0) \mbox{ exists and is finite. }\]
\end{lemma}
\begin{proof}
The idea is to use \cite[Corollary 1]{elling-ctlip} to obtain differentiability after restriction to a subsequence.  However, the result on which the Corollary depends is only true for functions from a subset of $\R^n$ to $\R^m$, with $n \geq m$, which need not be the case for our $V: \R \rightarrow \R^m$.  We instead apply the result to the following function from $E \subset \R$ to $\R$. 
For the background state $\Vb$ and $\xi \in I^\alpha$, define the function
\[ \xi \mapsto l^\alpha(\Vb)V(\xi).\]
 \cite{elling-ctlip} proves that for any $E \subset I^\alpha$, for almost all $\xi_0 \in E$ there exists $D'$ with $E \supset D' \ni \xi_0$ such that 
\[ \big( l^{\alpha}(\Vb)V _{|D'}\big) ' (\xi_0) \mbox{ exists and is finite.} \]
Recalling \myeqref{eq:Adiff}, we have
\[ \Big(\hat{A}(V(\xi),V(\xi_0)) - \xi_0 I \Big) \big(V(\xi)-V(\xi_0)\big) = \int_{\xi_0} ^{\xi} V(\xi)-V(\eta)\, d \eta .\]
 For $\beta \neq \alpha$ multiply $\hat{l}^\beta\big(V(\xi),V(\xi_0)\big)$ on the left to obtain
\[\Big(\hat{\lambda}^{\beta}(V(\xi),V(\xi_0))-\xi_0\Big)\hat{l}^{\beta}(V(\xi),V(\xi_0)) \big(V(\xi)-V(\xi_0)\big) 
= \hat{l}^{\beta}(V(\xi),V(\xi_0)) \int_{\xi_0} ^{\xi} V(\xi)-V(\eta)\, d \eta. \]
We then estimate, for $M$ only depending on $\Ph_\epsilon$:
\[M|\xi-\xi_0| \geq \big|\hat{\lambda}^{\beta}(V,V_0)-\xi_0 \big| \big|\hat{l}^{\beta}(V,V_0) \big(V(\xi)-V(\xi_0)\big)\big| .\]
\myeqref{eq:stricthyperhat} bounds $|\hat{\lambda}^{\beta}(V,V_0)-\xi_0|$ from 0; so, with some other constant $M'$ we have
\[M'|\xi-\xi_0| \geq \big|\hat{l}^{\beta}(V,V_0) \big(V(\xi)-V(\xi_0)\big)\big| .\]
Therefore, 
\[ \xi \mapsto \hat{l}^\beta \big(V(\xi), V(\xi_0)\big)\big(V(\xi)-V(\xi_0)\big) \]
(which is 0 at $\xi = \xi_0$) is Lipschitz at $\xi_0$ with constant $\leq M'$.  This implies that the difference quotients
\[ \frac{\hat{l}^\beta \big(V(\xi), V(\xi_0)\big)\big(V(\xi)-V(\xi_0)\big) - \hat{l}^\beta\big(V(\xi_0), V(\xi_0)\big)\big(V(\xi_0)-V(\xi_0)\big)}{\xi - \xi_0}\]
\[ = \frac{\hat{l}^\beta \big(V(\xi), V(\xi_0)\big)\big(V(\xi)-V(\xi_0)\big)}{\xi - \xi_0} \]
are contained in $B_{M'}(0)$, a compact set in $\R^m$  (for $\xi$ sufficiently close to $\xi_0$).  Therefore, for each $\beta \neq \alpha$ we can successively pass to nested subsequences in $D'$ so that we finally obtain $D$ such that
\[ \big(l^\alpha(\Vb)V_{|D} \big) ' (\xi_0), \qquad \big(\hat{l}^\beta\big(V,V(\xi_0)\big)V_{|D} \big) ' (\xi_0) \mbox{ exist and are finite}\]
(where $D \subset D' \subset E \subset I^\alpha$ and $\xi_0 \in D$).
We now claim that
\[ W \mapsto g^\alpha(W) := l^\alpha(\Vb)W\]
and 
\[W \mapsto g^\beta(W) := \hat{l}^\beta\big(W,V(\xi_0)\big)\big(W-V(\xi_0)\big) \qquad (\beta \neq \alpha) \]
yield a local diffeomorphism $\Ph_\epsilon \ni W \mapsto g(W) := \big(g^1(W),...,g^m(W)\big)$ (for possibly smaller $\epsilon$).
To see this, notice that for $\beta \neq \alpha$
\[0 = g^\beta_W \big(V(\xi_0)\big)z = l^\beta \big(V(\xi_0)\big)z  \quad\Rightarrow\quad z \parallel r^\alpha(V(\xi_0)).\]
Then,
\[ 0 = g^\alpha_W \big(V(\xi_0)\big)z = l^\alpha(\Vb)z= l^\alpha\big(V(\xi_0)\big)z + \bO(\epsilon)z \quad\Rightarrow\quad z=0.\]
Since $(g \circ V)_{|D}$ is differentiable at $\xi_0$ and $g$ is a local diffeomorphism, we have that $V_{|D}$ is differentiable at $\xi_0$, and the lemma is proved.
\end{proof}

\begin{theorem}
    \mylabel{th:lindeg}%
    On a linearly degenerate (forward or backward) sector, 
    $V$ is either constant, or constant on each side of a single contact discontinuity.
\end{theorem}
\begin{proof}
    By Lemma \myref{lem:contver}, 
    $F:=\{\xi\in I^\alpha~|~\xi=\lambda^\alpha(U(\xi))\}$ is closed and
    $V$ is constant on $I^\alpha\backslash F$.
    
    Assume there are $\xi_1,\xi_2\in F$ and $\eta\in I^\alpha$ with $\xi_1<\eta<\xi_2$.
    Then we can choose a maximal $\boi{\eta^-}{\eta^+}$ containing $\eta$ but not meeting $F$.
    Necessarily $\eta^\pm\in F$, so $\eta^+=\lambda^\alpha(V(\eta^+))$ and $\eta^-=\lambda^\alpha(V(\eta^-))$.
    But $V$ is constant on $\boi{\eta^-}{\eta^+}$, so $\eta^+=\eta^-$, which is a contradiction.

    Hence $F$ must be a closed interval.

    Assume $F$ has positive length. By \myeqref{eq:fintfull}, $f(V)-\xi V$ is Lipschitz and 
    therefore differentiable on $E \subset F$ (where $F \setminus E$ has measure zero) with
    \begin{alignat}{1}
   \big(f(V)-\xi V\big)_\xi + V = 0 \qquad \mbox{ on $E$.} \myeqlabel{eq:fyx} 
   \end{alignat}
   Note that this is in the strong sense, not just distributionally.
   By Lemma \myref{lem:subseq}, for almost every $\xi \in E$ we can find $D \subset E$ containing $\xi$ such that $V_{|D}$ is differentiable at $\xi$.  Thus if $F$ has positive length, there exists 
   $\xi \in D \subset E \subset F$ such that \myeqref{eq:fyx} holds and $V_{|D}$ is differentiable at $\xi$.  Therefore, we have
   \[\Big(f_V\big(V(\xi)\big)-\xi I\Big) \partial_\xi V_{|D}(\xi) = 0, \]
  so $\partial_\xi V_{|D}(\xi)\parallel r^\alpha(V(\xi))$, hence
    $$ \lambda^\alpha_V(V(\xi)) \partial_\xi V_{|D}(\xi) \topref{eq:lindeg}{=} 0 $$
    by linear degeneracy. 
    However, 
    $$ \xi = \lambda^\alpha(V(\xi)) $$
    implies
    $$ 1 = \lambda^\alpha_V(V(\xi)) \partial_\xi V_{|D}(\xi), $$
   which is a contradiction.
    
    Hence $F$ must be a point (or empty, which can but need not be ruled out). 
\end{proof}

\section{Genuinely nonlinear sectors}

Consider a genuinely nonlinear $I^\alpha$.
Consider either the forward or the backward sector.
We partition $I^\alpha$ into the three sets 
\begin{alignat}{1}
    \Sx &:= \{\xi\in I^\alpha~|~J(V;\xi)>0\}, \\
    \Rx &:= \{\xi\in I^\alpha~|~J(V;\xi)=0,~\xi=\lambda_k(V(\xi))\}, \\
    \Cx &:= \{\xi\in I^\alpha~|~J(V;\xi)=0,~\xi\neq\lambda_k(V(\xi))\},
\end{alignat}
where $\Sx$ stands for ``shock'', $\Rx$ for ``resonance'', $\Cx$ for ``constant''.
Complements (denoted by $\complement$) are taken with respect to $I^\alpha$.

\subsection{Backward sectors}

Consider a \textit{backward} sector ($\sgn(x)=-\sgn(e_{VV}r^\alpha r^\alpha)$).  Assume $V$ is admissible.

First we observe crucially that shocks of admissible $V$ must have a left and right neighbourhood in each of which $V$
is constant. The neighbourhood size is lower-bounded proportionally to the shock strength.

\begin{figure}[h]
        \centerline{\input{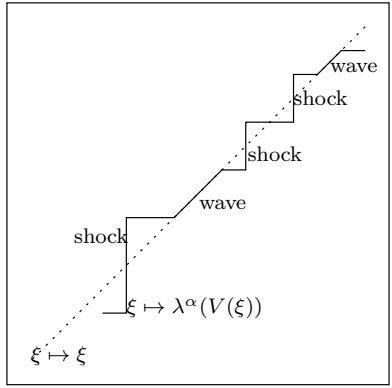}}
        \caption{For admissible $V$, each shock has a constant neighbourhood with lower size bound proportional to the shock strength.
            Reason: only in-admissible shocks could jump $\lambda(V(\xi))$ back to $\xi$ immediately. For backward sectors, consecutive shocks or shocks interspersed with compression waves are possible.}
        \mylabel{fig:backshock}
\end{figure}

\begin{theorem}
    \mylabel{th:shock-nbh}%
    For any $\xb\in \Sx$ there are $\sigma^+(\xb)>\xb$ (maximal) and $\sigma^-(\xb)<\xb$ (minimal)
    so that $V$ is constant on $\roi{\sigma^-(\xb)}{\xb},\loi{\xb}{\sigma^+(\xb)}\subset \overline{I^\alpha}$.
    Moreover
    $\sigma^\pm(\xb)\in \Rx\cup\partial I^\alpha$,
    \begin{alignat}{1}
        \sigma^-(\xb) &\leq \xb-\delta_LJ(V;\xb), \myeqlabel{eq:shocknbhL} \\
        \sigma^+(\xb) &\geq \xb+\delta_LJ(V;\xb), \myeqlabel{eq:shocknbhR}
    \end{alignat}
    and
    \begin{alignat}{1}
        &\lambda^\alpha(V(\xb+))-\xb \geq \delta_L J(V;\xb), \myeqlabel{eq:ulax2} \\
        &\lambda^\alpha(V(\xb-))-\xb \leq -\delta_L J(V;\xb) \myeqlabel{eq:llax2}
    \end{alignat}
    (where $\delta_L$ is as in \myeqref{eq:backward-lax}).
\end{theorem}

\begin{proof}
    (See Figure \myref{fig:backshock}.)

    1. Assume $V$ is discontinuous at $\xb$.
    Then we can choose a strictly decreasing sequence $(\xi^+_k)\downarrow\xb$ and another
    sequence $(\xi^-_k)\rightarrow\xb$ so that $\xi^-_k<\xi^+_k$ and $V(\xi^\pm_k)\rightarrow V^\pm$.
    The backward Lax condition \myeqref{eq:backward-lax} implies $\lambda^\alpha(V^+)-\xb>0$.

    Assume there is no $\delta>0$ so that $\lambda^\alpha(V(\xi))-\xi>0$ for $\xi\in\boi{\xb}{\xb+\delta}$.
    Then we can rename $(\xi^+_k)$ to $(\xi^-_k)$ and $V^+$ to $V^-$ (replacing the previous choice)
    and choose a new decreasing sequence $(\xi^+_k)\downarrow\xb$ 
    so that $\lambda^\alpha(V(\xi^+_k))-\xi^+_k\leq 0$ and so that $V(\xi^+_k)\rightarrow V^+$.
    We may assume, by omitting members from both sequences, that $\xi^-_k<\xi^+_k$ for all $k$.
    Then the backward Lax condition \myeqref{eq:backward-lax} yields $\lambda^\alpha(V^+)-\xb>0$, but that implies
    $\lambda^\alpha(V(\xi^+_k))-\xi^+_k>0$ for $k$ sufficiently large, which is a contradiction.

    Thus we may choose a maximal $\sigma^+(\xb)\in \overline{I^\alpha}\cap\boi{\xb}{\infty}$ so that
    \begin{alignat}{1}
        &\forall \xi\in\boi{\xb}{\sigma^+(\xb)} : \lambda^\alpha(V(\xi))-\xi > 0. \myeqlabel{eq:lax-upper}
    \end{alignat}
    Analogously we obtain a minimal $\sigma^-(\xb)\in \overline{I^\alpha}\cap\boi{-\infty}{\xb}$ so that 
    \begin{alignat}{1}
        &\forall \xi\in\boi{\sigma^-(\xb)}{\xb} : \lambda^\alpha(V(\xi))-\xi < 0. \myeqlabel{eq:lax-lower} 
    \end{alignat}
    
    2. If $\loi{\xb}{\sigma^+(\xb)}$ contained a $\xi\in \Sx$, then we could choose 
    \begin{alignat}{1}&
        \eta\in\boi{\xb}{\sigma^+(\xb)} \cap \boi{\sigma^-(\xi)}{\xi}
    \notag\end{alignat}
    so that
    \begin{alignat}{1}&
        \lambda^\alpha(V(\eta)) 
        \overset{\text{\myeqref{eq:lax-upper} for $\xb$}}{>} \eta 
        \overset{\text{\myeqref{eq:lax-lower} for $\xi$}}{>} \lambda^\alpha(V(\eta)),
    \notag\end{alignat}
    which is a contradiction. 
    Hence 
    \begin{alignat}{1}&
        \text{$V$ is continuous at every $\xi\in\loi{\xb}{\sigma^+(\xb)}$.} \myeqlabel{eq:Vcont}
    \end{alignat}
    By Theorem \myref{th:Uconst}, \myeqref{eq:lax-upper} combined with $\xi\neq\lambda^\beta(V(\xi))$ for $\beta\neq \alpha$ 
    (by definition of $I^\alpha$) yields 
    \begin{alignat}{1}&
        \text{$V$ is constant on $\loi{\xb}{\sigma^+(\xb)}$.} \myeqlabel{eq:Vconst}
    \end{alignat}
    Analogously we show $V$ is constant on $\roi{\sigma^-(\xb)}{\xb}$. 
    Then we may take any $(\xi^\pm_k)\rightarrow\xb$ with $\xi^-_k<\xb<\xi^+_k$ and 
    and $V(\xi^\pm_k)\rightarrow V^\pm$ and obtain
    \myeqref{eq:ulax2} and \myeqref{eq:llax2} from the backward Lax condition \myeqref{eq:backward-lax}.
    
    3. The boundary $\sigma^+(\xb)$ with the property \myeqref{eq:lax-upper} is maximal.
    If it is not a boundary point of $I^\alpha$, then 
    there is a sequence $(\eta_n)\downarrow\sigma^+(\xb)$ in $I^\alpha$ with
    \begin{alignat}{1}
        \lambda^\alpha(V(\eta_n)) - \eta_n &\leq 0.
    \notag\end{alignat}
    By \myeqref{eq:Vcont} that means
    \begin{alignat}{1}
        \lambda^\alpha(V(\sigma^+(\xb))) - \sigma^+(\xb) &\leq 0.
    \notag\end{alignat}
    On the other hand, \myeqref{eq:lax-upper} and \myeqref{eq:Vcont} show $<0$ is not possible, so
    \begin{alignat}{1}&
        \sigma^+(\xb) = \lambda^\alpha(V(\sigma^+(\xb))) 
    \notag\end{alignat}
    which means $\sigma^+(\xb)\in \Rx$, and
    \begin{alignat}{1}
        \sigma^+(\xb) = \lambda^\alpha(V(\sigma^+(\xb))) \topref{eq:Vconst}{=} \lambda^\alpha(V(\xb+))
        \topref{eq:ulax2}{\geq} \xb+\delta_L J(V;\xi)
        \notag\end{alignat}
    which implies \myeqref{eq:shocknbhR}. Analogously we obtain \myeqref{eq:shocknbhL}.
\end{proof}

\begin{remark}
    \mylabel{rem:Ddiscrete}%
    In particular $\Sx$ is discrete, hence countable.
    (This does not imply $V\in\BV(I^\alpha)$ yet until we also show the continuous part of $V$ has finite variation.)
\end{remark}

Since we have shown now that $V$ has well-defined left and right limits in each discontinuity, 
we may modify $V$ in each $\xi\in\Sx$ to be the \emph{right} limit. 

\begin{lemma}
    \mylabel{lem:Fshock}%
    There is a constant $C_\Sx$, independent of $V$, so that for any $\xb\in \Sx$,
    $$\xi\not\in\boi{\sigma^-(\xb)}{\sigma^+(\xb)}$$
    implies
    \begin{alignat}{1}
        J(V;\xb) , \big|\lambda^\alpha(V(\xb+))-\xb\big|, \big|\lambda^\alpha(V(\xb-))-\xb\big| \leq C_\Sx|\xi-\xb|. \myeqlabel{eq:lamlip}
    \end{alignat}
\end{lemma}
\begin{proof}
    (See Figure \myref{fig:FDseq}.)
    \begin{alignat}{1}&
        |\xi-\xb| \geq \min\big\{|\sigma^-(\xb)-\xb|,|\sigma^+(\xb)-\xb|\big\} 
        \overset{\text{\myeqref{eq:shocknbhL}}}{\underset{\text{\myeqref{eq:shocknbhR}}}{\geq}}
        \delta_LJ(V;\xb)
        \notag\\&
        \Rightarrow\quad
        J(V;\xb) \leq \delta_L^{-1}|\xi-\xb|
        \notag\\&
        \overset{\text{\myeqref{eq:ulax2},\myeqref{eq:llax2}}}{\Rightarrow}\quad
        |\lambda^\alpha(V(\xb+))-\xb|,|\lambda^\alpha(V(\xb-))-\xb| \leq \delta_LJ(V;\xb) \leq |\xi-\xb|.
    \notag\end{alignat}
    Take $C_\Sx\geq \max(1,\delta_L^{-1})$.
\end{proof}

\begin{lemma}
    \mylabel{lem:limitpoint}%
    If $\xi$ is a limit point of $\Sx$, then $\xi\in \Rx$.
\end{lemma}
\begin{proof}
    (See Figure \myref{fig:FDseq}.)
    Let $(\xi_n)\rightarrow\xi$ be a strictly decreasing sequence in $\Sx$ (the strictly increasing case is analogous).
    $\Sx$ is discrete, so $\xi\not\in\boi{\sigma^-(\xi_n)}{\sigma^+(\xi_n)}$ (it could not be a limit point otherwise).
    Choose some $\eta_n\in\boi{\sigma^-(\xi_n)}{\xi_n}$ for each $n$. Then
    \begin{alignat}{1}
        |\lambda^\alpha(V(\eta_n))-\eta_n| 
        &=
        |\lambda^\alpha(V(\xi_n-))-\eta_n|
        \\
        &\leq 
        |\lambda^\alpha(V(\xi_n-))-\xi_n| + |\xi_n-\eta_n|
        \notag\\&\topref{eq:lamlip}{\leq}
        C_\Sx|\xi_n-\xi| + |\xi_n-\xi|
        \overset{n\rightarrow\infty}{\rightarrow} 0.
    \notag\end{alignat}
    $(\eta_n)\rightarrow\xi$ and $\xi\not\in \Sx$, so $\lambda^\alpha\circ V$ is continuous at $\xi$ and therefore $\lambda^\alpha(V(\xi))=\xi$.
\end{proof}

\begin{theorem}
    \mylabel{th:constset}
    If $\xb\in \Cx$, then $V$ is constant on an interval
    $\boi{\kappa^-(\xb)}{\kappa^+(\xb)}$ that contains $\xb$.
    We take the interval maximal in $I^\alpha$.
    $\kappa^\pm(\xb)$ are either in $\Rx\cup \Sx$ or endpoints of $I^\alpha$.
\end{theorem}
\begin{proof}
    By Lemma \myref{lem:limitpoint}, $\xb$ is not a limit point of $\Sx$ (since it would be in $\Rx$ otherwise,
    and $\Rx\cap \Cx=\emptyset$. Hence $V$ is continuous in a neighbourhood of $\xb$.
    Then $\lambda^\alpha(V(\xb))-\xb\neq 0$ implies $\lambda^\alpha(V(\xi))-\xi\neq 0$ for $\xi$ in a neighbourhood of $\xb$. 
    Since $\lambda^\beta(V(\xi))-\xi\neq 0$ for $\beta\neq\alpha$ by definition of $I^\alpha$,
    Theorem \myref{th:Uconst} shows $V$ is constant on this neighbourhood.
    We may take $\boi{\kappa^-(\xb)}{\kappa^+(\xb)}$ as described in the statement.

    By what we have already shown, $\kappa^\pm(\xb)\not\in \Cx$ because it would violate their extremality.
\end{proof}

\begin{figure}[h]
        \centerline{\input{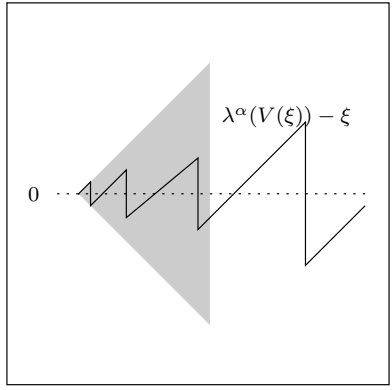}}
        \caption{$\xi\mapsto\lambda(V(\xi))-\xi$ is Lipschitz at points $\xb$ where it is $0$, since shocks have 
            to weaken at least proportionally to their distance from $\xb$.}
        \mylabel{fig:FDseq}
\end{figure}

From now until just before Theorem \myref{th:VSL}, we consider only $V$ restricted to $I^\alpha$, without writing $V_{|I^\alpha}$ to avoid clutter.

\begin{lemma}
    \mylabel{lem:lamlip}
    If $\xb\in \Rx$, then $\lambda^\alpha\circ V_{|\complement \Sx}$ is Lipschitz at $\xb$ with Lipschitz constant $\leq C_\Sx+2$.
\end{lemma}
\begin{proof}
    (See Figure \myref{fig:FDseq}.)
    Consider $\xi\not\in \Sx$ with $\xi>\xb$ (the case $\xi<\xb$ is analogous).

    We first obtain a Lipschitz estimate for $\lambda^\alpha(V(\xi))-\xi$.

    1. If $\xi\in \Rx$, then by definition of $\Rx$
    \begin{alignat}{1}&
        \big| \lambda^\alpha(V(\xi))-\xi \big| = 0 .
        \myeqlabel{eq:lip1}
    \end{alignat}
    
    2. If $\xi\in \Cx$, then $\kappa^-(\xi)\in \Rx\cup \Sx$ by Theorem \myref{th:constset}.
    ($\kappa^-(\xi)$ cannot be a boundary point of $I^\alpha$ since $\xb\leq\kappa^-(\xi)<\xi$.)

    2a. If $\kappa^-(\xi)\in \Rx$, then $\lambda^\alpha(V(\kappa^-(\xi)))=\kappa^-(\xi)$, so
    \begin{alignat}{1}&
        \lambda^\alpha(V(\xi)) = \lambda^\alpha\big(V(\kappa^-(\xi))\big) = \kappa^-(\xi) \in \cli{\xb}{\xi}
        \notag\\\Rightarrow&
        \big|\lambda^\alpha(V(\xi)) - \xi \big| 
        =
        \big| \kappa^-(\xi) - \xi \big| 
        \leq
        |\xb-\xi|.
        \myeqlabel{eq:lip2}
    \end{alignat}
    
    2b. If $\kappa^-(\xi)\in \Sx$, then       
    (note $\xb\not\in\boi{\sigma^-(\kappa^-(\xi))}{\sigma^+(\kappa^-(\xi))}\subset\Cx$ since $\xb\in\Rx$, $\Rx\cap\Cx=\emptyset$,
    so Lemma \myref{lem:Fshock} applies to $\xb$)
    \begin{alignat}{1}&
        \big|\lambda^\alpha(V(\xi))-\xi\big| 
        = \big|\lambda^\alpha\big(V(\kappa^-(\xi)+)\big)-\xi\big| 
        \notag\\&\leq \big|\lambda^\alpha\big(V(\kappa^-(\xi)+)\big)-\kappa^-(\xi)\big| + |\kappa^-(\xi)-\xi|
        \notag\\&\overset{\text{Lemma \myref{lem:Fshock}}}\leq C_\Sx|\kappa^-(\xi)-\xb| + |\kappa^-(\xi)-\xi| 
        \leq (C_\Sx+1)|\xb-\xi|. \myeqlabel{eq:lip3}
    \end{alignat}
    
    Combining all cases \myeqref{eq:lip1}, \myeqref{eq:lip2}, \myeqref{eq:lip3} we see that 
    \begin{alignat}{1}
        \big| \lambda^\alpha(V(\xi))-\xi - \subeq{\big(\lambda^\alpha(V(\xb))-\xb\big)}{=0} \big| 
        = |\lambda^\alpha(V(\xi))-\xi| 
        &\leq (C_\Sx+1) |\xb-\xi|.
    \notag\end{alignat}
Hence $\lambda^\alpha\circ V_{|\complement \Sx}$ itself is also Lipschitz at $\xb$, with constant $\leq C_\Sx+2$.
\end{proof}

\begin{theorem}
    \mylabel{th:lipF}%
    $V$ is Lipschitz with constant $\leq C_\Rx$ (independent of $V$) at any $\xb\in \Rx$.
\end{theorem}
\begin{proof}
    For each $\beta\neq\alpha$ multiply $\hat\ell_\beta$ 
    onto \myeqref{eq:Adiff} to obtain, for some $M$ depending only on $\Peps$.
    \begin{alignat}{1}
        M|\xi-\xb|
        &\geq
        \Big|\hat\ell^\beta\big(V(\xi),V(\xb)\big) \Big(\hat A\big(V(\xi),V(\xb)\big) - \xb I\Big) \big(V(\xi)-V(\xb)\big)\Big|
        \notag\\&=
        \Big|\Big( \hat\lambda^\beta\big(V(\xi),V(\xb)\big) - \lambda^\alpha\big(V(\xb)\big) \Big) 
        \hat\ell^\beta\big(V(\xi),V(\xb)\big) 
        \big(V(\xi)-V(\xb)\big)\Big|
    \notag\end{alignat}
    Since $\hat\lambda^\beta-\lambda^\alpha$ is bounded away from $0$ by \myeqref{eq:stricthyperhat}, we obtain
    for some other constant $M'$ that
    \begin{alignat}{1}
        M'|\xi-\xb|
        &\geq
        \Big|\hat\ell^\beta\big(V(\xi),V(\xb)\big) 
        \big(V(\xi)-V(\xb)\big)\Big|
    \notag\end{alignat}
    so that
    $$ \xi\mapsto\hat\ell^\beta\big(V(\xi),V(\xb)\big)\big(V(\xi)-V(\xb)\big) $$
    (which is $=0$ at $\xi=\xb$) is Lipschitz at $\xb$ with constant $\leq M$.

    $$ W \mapsto g^\beta(W) := \hat\ell^\beta\big(W,V(\xb)\big)\big(W-V(\xb)\big) \qquad (\beta\neq\alpha)$$
    and 
    $$ W \mapsto g^\alpha(W) := \lambda^\alpha(W) $$
    yield a local diffeomorphism $\Peps\ni W\mapsto g(W):=(g^1(W),...,g^m(W))$
    (after reducing $\ueps>0$, if necessary): for $\beta\neq\alpha$,
    $$ 0 = g^\beta_W(V(\xb)) z = \ell^\beta(V(\xb)) z $$
    implies $z\parallel r^\alpha(W)$, and then 
    \begin{alignat}{1}&
        0 = g_W^\alpha(V(\xb)) z =\lambda^\alpha_W(V(\xb)) z \quad\topref{eq:gennon}{\Rightarrow} \quad z = 0
        \myeqlabel{eq:fW}
    \end{alignat}
    by genuine nonlinearity; hence $g_W(V(\xb))$ is regular.
    
    Lipschitz continuity at $\xb$ for $\complement \Sx\ni\xi\mapsto g(V(\xi))$ 
    (from Lemma \myref{lem:lamlip} for $g^\alpha$) 
    implies Lipschitz continuity at $\xb$ for $\complement \Sx\ni\xi\mapsto V(\xi)$.
\end{proof}

\begin{lemma}
    \mylabel{lem:VS}
    Define
    \begin{alignat}{1}
        V_S(\xi) &:= \sum_{\eta\in \Sx,~\eta<\xi}\big(V(\eta+)-V(\eta-)\big). \myeqlabel{eq:VSdef}
    \end{alignat}
    Then $V_S$ is well-defined and a right-continuous saltus function (see Definition \myref{def:saltus} in the Appendix). 
\end{lemma}
\begin{proof}
    \begin{alignat}{1}&
        \sum_{\eta\in \Sx}|V(\eta+)-V(\eta-)|
        =
        \sum_{\eta\in \Sx}J(V;\eta)
        \notag\\&
        \overset{\text{\myeqref{eq:shocknbhL}}}{\underset{\text{\myeqref{eq:shocknbhR}}}{\leq}}
        (2\delta_L)^{-1}\sum_{\eta\in \Sx}|\sigma^+(\eta)-\sigma^-(\eta)|
        \leq
        (2\delta_L)^{-1}|I^\alpha| < \infty
        \notag\end{alignat}
    since the neighbourhoods $\boi{\sigma^-(\eta)}{\sigma^+(\eta)}$ of $\eta\in \Sx$ are pairwise disjoint
    and contained in $I^\alpha$.
    Hence not only is $\Sx$ countable, but the jumps sum to a finite number. Hence \myeqref{eq:VSdef} makes sense.
    In Definition \myref{def:saltus} only $b_n$ are used so that $V_S$ is right-continuous.
\end{proof}

\begin{lemma}
    \mylabel{lem:VSlipF}
    $V_S$ is Lipschitz with constant $\leq C_S$ ($C_S$ independent of $V$) at any $\xb\in \Rx$.
\end{lemma}
\begin{proof}
    Let $\xi>\xb$ (the case $\xi<\xb$ is analogous). 

    1. If $\xi\not\in\boi{\sigma^-(\eta)}{\sigma^+(\eta)}$ for some $\eta\in\Sx$, then we may estimate
    \begin{alignat}{1}&
        |V_S(\xi)-V_S(\xb)|
        \leq
        \sum_{\eta\in \Sx,~\xb\leq\eta<\xi} J(V;\eta)
        \overset{\text{\myeqref{eq:shocknbhL}}}{\underset{\text{\myeqref{eq:shocknbhR}}}{\leq}}
        (2\delta_L)^{-1}\sum_{\eta\in \Sx,~\xb\leq\eta<\xi}|\sigma^+(\eta)-\sigma^-(\eta)|
        \notag\\&\leq
        (2\delta_L)^{-1}|\xi-\xb|
        \myeqlabel{eq:notxi}
    \end{alignat}
    since the neighbourhoods $\boi{\sigma^-(\eta)}{\sigma^+(\eta)}$ of distinct $\eta\in \Sx$ 
    are pairwise disjoint and contained in $\cli{\xi}{\xb}$.

    2. If $\xi\in\boi{\sigma^-(\eta)}{\sigma^+(\eta)}$ for some $\eta\in\Sx$, 
    then we apply \myeqref{eq:notxi} with $\xi\leftarrow\sigma^-(\eta)$:
    \begin{alignat}{1}
        |V_S(\xi)-V_S(\xb)|
        &\leq
        |V_S(\xi)-V_S(\sigma^-(\eta))| + |V_S(\sigma^-(\eta))-V_S(\xb)|
        \notag\\&\topref{eq:notxi}{\leq}
        |V_S(\xi)-V_S(\sigma^-(\eta))| + (2\delta_L)^{-1}|\sigma^-(\eta)-\xb|
        \notag\\&\leq
        |V_S(\xi)-V_S(\sigma^-(\eta))| + (2\delta_L)^{-1}|\xi-\xb|.
    \notag\end{alignat}

    2a. For $\xi\in\boi{\sigma^-(\eta)}{\eta}$ the first term is $=0$.

    2b. For $\xi\in\roi{\eta}{\sigma^+(\eta)}$ the first term is estimated by Lemma \myref{lem:Fshock}
    (using $\xb\not\in\boi{\sigma^-(\eta)}{\sigma^+(\eta)}\subset\Cx$ since $\xb\in\Rx$, $\Rx\cap\Cx=\emptyset$):
    \begin{alignat}{1}&
        |V_S(\xi)-V_S(\sigma^-(\eta))| = J(V;\eta) \topref{eq:lamlip}{\leq} C_\Sx|\eta-\xb| \leq C_\Sx|\xi-\xb|.
    \notag\end{alignat}
    
    Altogether we get the desired estimate, with $C_S:=C_\Sx+(2\delta_L)^{-1}$.
\end{proof}

\begin{theorem}
    \mylabel{th:VSL}
    $V=V_S+V_L$ 
    where $V_L$ is Lipschitz, with a Lipschitz constant independent of $V$.
    In particular $V$ is $\BV$.
\end{theorem}
\begin{proof}
    It is sufficient to obtain a Lipschitz estimate for $\xi,\eta\in I^\alpha$ ($\xi<\eta$) since $V$ is constant 
    in between intervals $I^\beta$, the distance to $I^\beta$ for $\beta\neq\alpha$ has a positive lower bound independent of $V$,
    and $V$ is bounded.

    1. First consider $\xi\in \Rx$. 
    \begin{alignat}{1}
        |V_L(\eta) - V_L(\xi)| 
        &\leq 
        |V(\eta) - V(\xi)|
        +
        |V_S(\eta) - V_S(\xi)|
        \notag\\&\overset{\text{Theorem \myref{th:lipF}}}{\underset{\text{Lemma \myref{lem:VSlipF}}}{\leq}}
        C |\eta-\xi|
        \myeqlabel{eq:VL1}
    \end{alignat}
    for some constant $C$ independent of $V$.

    2. Now consider $\xi\in \Sx$. 
    Then $V_L$ is constant on $\boi{\sigma^-(\xi)}{\sigma^+(\xi)}$ (the jump of $V$ at $\xi$ is cancelled by $V_S$),
    so we only need a Lipschitz estimate for $\eta\geq\sigma^+(\xi)$
    (which implies $\sigma^+(\xi)\not\in\partial I^\alpha$).
    By Theorem \myref{th:shock-nbh}, $\sigma^+(\xi)\in \Rx$, so we may use \myeqref{eq:VL1} 
    (with $\xi\leftarrow\sigma^+(\xi)$) and $V_L(\xi)=V_L(\sigma^+(\xi))$ to get
    \begin{alignat}{1}
        |V_L(\xi)-V_L(\eta)| 
        &=
        |V_L(\sigma^+(\xi))-V_L(\eta)|
        \topref{eq:VL1}{\leq}
        C |\sigma^+(\xi)-\eta|
        \leq 
        C |\xi-\eta|.
        \myeqlabel{eq:VL2}
    \end{alignat}

    3. Finally consider $\xi\in \Cx$. 
    Then $V_L$ (like $V$) is constant on $\boi{\kappa^-(\xi)}{\kappa^+(\xi)}$,
    so we only need a Lipschitz estimate for $\eta\geq\kappa^+(\xi)$
    (which implies $\kappa^+(\xi)\not\in\partial I^\alpha$).
    By Theorem \myref{th:constset}, $\kappa^+(\xi)\in \Rx\cup \Sx$. 

    3a. For $\kappa^+(\xi)\in \Rx$ we may use \myeqref{eq:VL1} with $\xi\leftarrow\kappa^+(\xi)\in \Rx$ and $V_L(\xi)=V_L(\kappa^+(\xi))$ to get
    \begin{alignat}{1}
        |V_L(\xi)-V_L(\eta)|
        &=
        |V_L(\kappa^+(\xi))-V_L(\eta)|
        \topref{eq:VL1}{\leq}
        C |\kappa^+(\xi)-\eta|
        \leq 
        C |\xi-\eta|.
        \myeqlabel{eq:VL3a}
    \end{alignat}

    3b. For $\kappa^+(\xi)\in \Sx$ we may use \myeqref{eq:VL2} with $\xi\leftarrow\kappa^+(\xi)\in \Sx$ 
    and $V_L(\xi)=V_L(\kappa^+(\xi))$ to get
    \begin{alignat}{1}
        |V_L(\xi)-V_L(\eta)|
        &=
        |V_L(\kappa^+(\xi))-V_L(\eta)|
        \topref{eq:VL2}{\leq}
        C |\kappa^+(\xi)-\eta|
        \leq 
        C |\xi-\eta|.
        \myeqlabel{eq:VL3b}
    \end{alignat}
\end{proof}

\begin{remark}
This shows that entropy-admissible self-similar weak solutions to the Riemann problem (for sufficiently small jump) for $1-d$ hyperbolic conservation laws are unique in $L^{\infty}$ (assuming $||U(\cdot)-\Ub||_{L^{\infty}}$ sufficiently small), extending the well-known result that they are unique in $BV$ (see Theorem 9.4.1 in \cite{dafermos-book}).
\end{remark}

\subsection{Continuity on open nonempty intervals}

\begin{theorem}
    \mylabel{th:onerf}%
    Consider any genuinely nonlinear sector, forward or backward.
    If $V$ is continuous on an open interval $B\subset I^\alpha$, 
    then it is either constant or constant on either side of a single $\alpha$-simple wave.
\end{theorem}
\begin{proof}
    (See Figure \myref{fig:forsimple}.)
    $V$ is Lipschitz on $B$, since we can repeat Lemma \myref{lem:lamlip} and Theorem \myref{th:lipF}
    with obvious changes to their proofs ($\Sx$ need not be considered since $V$ is continuous here).

    By continuity of $V$, $\Cx\cap B$ is open, hence a countable union of disjoint open intervals. 
    $V$ is constant on each of these intervals, by Theorem \myref{th:Uconst}, 
    and so is $\lambda^\alpha\circ V$, so that $\lambda^\alpha(V(\xi))-\xi=0$ cannot be satisfied at both endpoints. 
    Therefore at least one endpoint of each of these intervals is not in $\Rx$. Suppose this endpoint is in $\Sx$.  
    Since $\Sx \cap B$ is empty, this endpoint is an endpoint of $B$.  
    If this endpoint is not in $\Sx$, then it still must be an endpoint of $B$. 
    Since there are only two endpoints, $\Cx\cap B$ is a union of at most two of these intervals. 
    
    It follows that $\Cx \cap B$ is either $B$ itself, $B$ minus a single point 
    (which by continuity of $V$ implies $V$ is constant on all of $B$), or $B \setminus (\Rx \cap B)$, 
    where $\Rx \cap B$ is a closed interval of positive length. 
    By definition of $\Rx$,
    \begin{alignat}{1}
        \lambda^\alpha(V(\xi)) &= \xi \myeqlabel{eq:lamVr}
    \end{alignat}
    on $\Rx\cap B$.
    By \myeqref{eq:claw-V-diffed} $V_\xi$ (defined a.e., since $V$ is Lipschitz) is a multiple of $r^\alpha(V(\xi))$.
    Therefore $\xi\mapsto V(\xi)$ is part of the $\alpha$-simple wave curve $R^\alpha$,
    and \myeqref{eq:lamVr} shows it is the $\xi$-parametrization of $R^\alpha$. 
    Hence $V$ is an $\alpha$-simple wave on $\Rx$.
\end{proof}

\begin{remark}
This shows that although infinitely many waves can occur in a backward sector, there cannot be consecutive simple waves.  For more than one simple wave to exist, there must be at least one shock in between.
\end{remark}

\subsection{Admissible genuinely nonlinear forward sectors}

\begin{figure}[h]
    \parbox[t]{.45\linewidth}{%
        \centerline{\input{forshock.pstex_t}}
        \caption{In a forward sector $\xi\mapsto\lambda(V(\xi))-\xi$ cannot return to $0$ after a shock,
            and has the wrong sign for another admissible shock.}
        \mylabel{fig:forshock}}
\hfill
   \parbox[t]{.45\linewidth}{%
        \centerline{\input{forsimple.pstex_t}}
        \caption{In a forward sector, after a simple wave $\xi\mapsto\lambda(V(\xi))-\xi$ has the wrong sign 
            for an admissible shock, so it cannot return to $0$.}
        \mylabel{fig:forsimple}}
\end{figure}

\begin{theorem}
    \mylabel{th:forward-shock}%
    Consider an admissible genuinely nonlinear forward sector.
    Then $V$ is either constant, or constant on either side of a single simple wave, or constant on either side 
    of a single shock.
\end{theorem}
\begin{proof}
    (See Figures \myref{fig:forshock} and \myref{fig:forsimple}.)
    Assume that $V$ is discontinuous at some $\xb\in I^\alpha$.
    Choose $(\xi_k^-),(\xi^+_k)\rightarrow\xb$ with $\xi^-_k<\xi^+_k$ and $V(\xi_k^\pm)\rightarrow V^\pm$
    where $[V]\neq 0$.
    The forward Lax condition \myeqref{eq:forward-lax} yields 
    \begin{alignat}{1}
        \lambda^\alpha(V^-) > \xb > \lambda^\alpha(V^+) \myeqlabel{eq:forlax}
    \end{alignat}
    We may proceed in the same manner as in the proof of Theorem \myref{th:shock-nbh}.
    $\lambda(V(\xi))$ is still constant in $\boi{\xb}{\sigma^+(\xb)}$ and $\xi$ is strictly increasing,
    but now \myeqref{eq:forlax} has the opposite comparisons:
    $\lambda^\alpha(V(\xi+))-\xi$ is \emph{negative} and cannot reach $0$ or change signs again. Hence $\sigma^+(\xb)$ is
    the right boundary of $I^\alpha$. By an analogous argument on the $\xi<\xb$ side we obtain that
    $\sigma^-(\xb)$ is the left boundary of $I^\alpha$.

    Now assume $V$ is continuous on $I^\alpha$. Then Theorem \myref{th:onerf} yields the rest of the result.
\end{proof}

\section{Isentropic Euler}

\subsection{Calculations}

We now focus on a particularly important case, the \emph{isentropic Euler equations}:
\begin{alignat}{1} &
    U_t + f^x(U)_x + f^y(U)_y = 0, 
    \notag\\&
    U = \begin{bmatrix}
        \rho \\
        m \\
        n
    \end{bmatrix}
    , \quad 
    f^x(U) = \left( \begin{array}{c}  m \\ m^2 \rho^{-1} + p \\ mn\rho^{-1} \end{array} \right), \quad
    f^y(U) = \left( \begin{array}{c} n \\ mn\rho^{-1} \\  n^2\rho^{-1} + p \end{array} \right)\quad.
    \notag\end{alignat}
Here $(m,n)$ is the momentum density vector, $\vec v=(u,v)=(\frac{m}{\rho},\frac{n}{\rho})$ the velocity. 
Then $P$ is an open subset of $\{(\rho,m,n)\in\R^3:\rho>0\}$.
We assume the \emph{pressure} $p=p(\rho)$ satisfies 
$$ c^2 = p'(\rho) > 0 $$
for all $\rho>0$; $c$ is the \emph{sound speed}.
We assume
\begin{alignat}{1}
    c_\rho &> -1 \myeqlabel{eq:crho}
\end{alignat}
which is satisfied for most relevant pressure laws, including $p(\rho)=\rho^\gamma$ for $\gamma>-1$.

Take
\begin{alignat}{1} 
    e(\rho) &= \int_0^\rho \frac{p(\rho)}{\rho^2}d\rho,
\notag\end{alignat}
then
\begin{alignat}{1}
    \eta(U) &:= \rho\big(e(\rho)+\frac12|\vec v|^2\big) , \quad \vec\psi(U) = (\eta(U)+p) \vec v
\notag\end{alignat}
form an entropy-flux pair $(\eta,\vec\psi)$ with uniformly convex $\eta$.

For simplicity we assume units have been chosen so that $c=1$ for $\rho=1$.

The Euler equations are invariant under rotation (and mirror reflection): if $U$ is a weak/weak entropy/strong solution, 
then for any $2\times 2$ orthogonal matrix $Q$, 
$$ 
U' = (\rho',\vec v'),\quad 
\rho'(t,\vec x')=\rho(t,\vec x),\quad
\vec v'(t,\vec x')=Q\vec v(t,\vec x),\quad 
\vec x'=Q\vec x 
$$ 
is another weak/weak entropy/strong solution. The equation also also invariant under change of inertial frame: for any $\vec a\in\R^2$, 
another solution $U''$ is
$$ U''=(\rho'',\vec v''),\quad
\rho''(t,\vec x'')=\rho(t,\vec x),\quad
\vec v''(t,\vec x'')=\vec v(t,\vec x)+\vec a,\quad 
\vec x''=\vec x + t\vec a. $$

Consider steady self-similar solutions. 
In the framework of the present paper we consider only the strictly hyperbolic case.
To this end we consider a background state $\Ub=(\rho_0,M_0,0)$ with $M_0>1$.
(Due to rotation invariance no generality is lost.
If we interpret supersonic steady Euler flow as an initial-value problem,
with data imposed at $x=-\infty$, hyperbolicity with $x$ as time and $y$ as space variable
requires $M>1$, not just $|\vec M|>1$.) 

In addition we choose $\ueps>0$ so small that $\|U-\Ub\|<\ueps$ implies
$M>1$ as well. We may also choose units so that $\rho_0=c_0=1$.

\begin{alignat}{1}
    f^x_U(U) &= \begin{bmatrix}
        0 & 1 & 0 \\ 
        -\frac{m^2}{\rho^2} + c^2 & \frac{2m}{\rho} & 0 \\ 
        -\frac{mn}{\rho^2} & \frac{n}{\rho} & \frac{m}{\rho} 
    \end{bmatrix} \qquad
    f^y_U(U) = \begin{bmatrix} 
        0 & 0 & 1 \\ 
        -\frac{mn}{\rho^2} & \frac{n}{\rho} & \frac{m}{\rho} \\ 
        -\frac{n^2}{\rho^2} + c^2 & 0 & \frac{2n}{\rho} 
    \end{bmatrix}  
\notag\end{alignat}
The eigenvalues of $f^x_U(1, M_0, 0)$ are $M_0 \pm 1, M_0$.  Therefore, if $M_0 > 1$ as assumed, all eigenvalues of $f^x_U$ will be positive, making $e_{VV}$ positive definite, and therefore the forward sectors have $x>0$.  If instead $M_0 < -1$ (as required for hyperbolicity with $-x$ serving as a time variable), then all the forward sectors would have $x<0$.  

The generalized eigenvalues (roots of $p$ in \myeqref{eq:hom}) are 
\begin{alignat}{1} 
    \lambda_\pm &= 
    \frac{mn \pm \rho c \sqrt{m^2+n^2-(\rho c)^2}}{m^2-(\rho c)^2}, \qquad
    \lambda_0 = 
    \frac{n}{m},
    \notag\end{alignat}
which are real, distinct and analytic functions of $U$ for $M>1$.

The generalized eigenvector $r_0$ for $\lambda_0$ is $(0,m,n)$. 
\begin{alignat}{1}
    \nabla_U\lambda_0(U)\cdot r_0(U) 
    &=
    \nabla_{(\rho,M,N)}\big(\frac{n}{m}\big)\cdot(0,m,n)
    =
    \big(0,\frac{-n}{m^2},\frac{1}{m}\big)\cdot(0,m,n) = 0,
    \notag\end{alignat}
so the $0$-field is linearly degenerate.

For the $\pm$-fields it is sufficient to consider the generalized eigenvectors only at $\Ub$:
\begin{alignat}{1}
    f^x(\Ub) &= \begin{bmatrix}
        0 & 1 & 0 \\ 
        1-m^2 & 2m & 0 \\ 
        -mn & n & m
    \end{bmatrix}, \quad
    f^y_U(\Ub) = \begin{bmatrix} 
        0 & 0 & 1 \\ 
        -mn & n & m \\ 
        1-n^2 & 0 & 2n
    \end{bmatrix} ,\quad
    r_\pm = \begin{bmatrix}
        \pm m  \\
        \pm (m^2-1) \\
        \sqrt{m^2-1}
    \end{bmatrix},
    \notag\end{alignat}
\begin{alignat}{1}
    \frac{\partial\lambda_\pm}{\partial n} 
    &=
    \frac{m \pm \rho c n (m^2+n^2-(\rho c)^2)^{-1/2}}{m^2-(\rho c)^2}
    \overset{\rho=c=1,n=0}=
    \frac{M}{M^2-1}.
\notag\end{alignat}
For the $\rho,m$ derivatives we may substitute $n=0$ first:
\begin{alignat}{1}
    \lambda_\pm &= 
    \frac{ \pm \rho c }{\sqrt{m^2-(\rho c)^2}}
    =
    \frac{ \pm 1 }{\sqrt{M^2-1}}.
    \myeqlabel{eq:lam-pm}
\end{alignat}
Then
\begin{alignat}{1}&
    \frac{\partial\lambda_\pm}{\partial m} 
    =
    -\frac12 2m\frac{ \pm \rho c }{(m^2-(\rho c)^2)^{3/2}}
    \overset{\rho=c=1,n=0}=
    \mp M(M^2-1)^{-3/2}, 
    \notag\\&
    \frac{\partial\lambda_\pm}{\partial \rho} 
    =
    \frac{\partial(\rho c)}{\partial\rho} \partial_{(\rho c)}\frac{ \pm \rho c }{\sqrt{m^2-(\rho c)^2}} 
    \notag\\&=
    \frac{\partial(\rho c)}{\partial\rho} 
    \Big( 
    \frac{ \pm 1 }{\sqrt{m^2-(\rho c)^2}} 
    -\frac12\cdot(-2\rho c)
    \frac{ \pm \rho c }{(m^2-(\rho c)^2)^{3/2}} 
    \Big)
    \notag\\&\overset{\rho=c=1,n=0}=
    \pm
    (\rho c)_{\rho|\rho=c=1}
    \Big( 
    \frac{ M^2 - 1 }{(M^2-1)^{3/2}} 
    +
    \frac{ 1 }{(M^2-1)^{3/2}} 
    \Big)
    \notag\\&=
    \pm
    (1+c_\rho(1))
    \frac{ M^2 }{(M^2-1)^{3/2}} .
    \notag\end{alignat}
Therefore, 
\begin{alignat}{1}
    \nabla_U\lambda_\pm (\Ub) \cdot r_\pm(\Ub)
    &=
    (M^2-1)^{-3/2} 
    \begin{bmatrix}
        \pm M^2(1+c_\rho(1)) \\
        \mp M \\
        M(M^2-1)^{1/2}
    \end{bmatrix}
    \cdot 
    \begin{bmatrix}
        \pm M \\
        \pm (M^2-1) \\
        \sqrt{M^2-1}
    \end{bmatrix}
    \notag\\&=
    \frac{
        M^3(1+c_\rho(1))
        -M(M^2-1)
        +M(M^2-1)
    }{(M^2-1)^{3/2}}
    = 
    \frac{M^3(1+c_\rho(1))}{(M^2-1)^{3/2}},
    \notag\end{alignat}
so by \myeqref{eq:crho} the $\pm$-fields are genuinely nonlinear at $\Ub$. If we choose $\ueps>0$ sufficiently small, then
they are genuinely nonlinear for all values of $U$ with $\|U-\Ub\|<\ueps$.

Consider the upper left quadrant, $y>0>x$, 
Here $\lambda_-$ is relevant. 
In increasing $x$ direction with fixed $y$, corresponding to decreasing $\xi$, 
the change of $U$ in a simple wave is given by $-r_-$: density increases, velocity turns down and decreases
(same effect as the $\lambda_-$-shocks).
This is a \defm{compression wave}. 
It can be approximated as the limit of an increasingly fine fan of weakening shocks.

In the upper right quadrant $x,y>0$, $\lambda_+$ is important.
In increasing $x$ direction with fixed $y$, corresponding to decreasing $\xi$, 
the change of $U$ in a simple wave is given by $-r_+$: density \emph{de}creases, velocity turns downwards and \emph{in}creases
(opposite to the behaviour of $\lambda_+$-shocks).
This is an \defm{expansion wave} (also known as \defm{Prandtl-Meyer} wave).

\subsection{Summary}

All results combined, we have the following description of steady and self-similar Euler flows $U$ that are sufficiently
$L^\infty$-close to a constant background state $\Ub=(\rho,Mc,0)$
with Mach number $M>1$ (supersonic), defining \defm{Mach angle} $\mu=\arcsin\frac{1}{M}
$  (see Figure \myref{fig:test}): 

1. they are necessarily $\BV$,

2. they are constant outside six narrow sectors whose center lines are $(1:0)$, $(\cos\mu:\sin\mu)$, $(\cos\mu:-\sin\mu)$,

3. in the $(1:0)$ forward and backwards sectors $U$ is constant on each side of a single contact discontinuity (which may vanish),

4. in the forward $(\cos\mu:\pm\sin\mu)$ sectors $U$ is constant on each side of a single shock or single rarefaction wave
(which may vanish),

5. in the backward $(\cos\mu:\pm\sin\mu)$ sectors $U$ can have an infinite or any finite number of shocks and compression waves, 
but

5a. two consecutive compression waves with a gap are not possible, and

5b. the shock set (on the unit circle) is discrete, with each shock having constant neighbourhoods on each side
whose size is lower-bounded proportionally to the shock strength.

It does not seem possible to improve these results without making additional assumptions. 
Examples with infinitely many consecutive shocks, or shocks interspersed with compression waves,
or compression waves ending in a point that is a limit point of shocks, can be constructed.

\section{Appendix}

\subsection{Saltus functions}

\begin{definition}
    \mylabel{def:saltus}%
    A \defm{saltus function} $f:D\rightarrow\R^m$ ($D\subset\R$) has the form
    $$ f(x) = \sum_{x_n\leq x}a_n + \sum_{x_n<x}b_n $$
    where $(x_n)$ is a sequence and $\sum a_n,\sum b_n$ are absolutely converging series.
\end{definition}

\subsection{Versions}

\begin{lemma}
    \mylabel{lem:version-adapt}%
    Let $\Omega\subset\R^n$ measurable nonempty, $K\subset\R^m$ compact, $U\in L^\infty(\Omega)$ so that $U(x)\in K$ for a.e.\ $x\in\Omega$,
    $g:K\rightarrow\R^k$ and $\tilde g:\Omega\rightarrow\R^k$ continuous. 
    If
    \begin{alignat}{1}
        g(U(x)) &\leq \tilde g(x) \quad\text{for a.e.\ $x\in\Omega$,} \myeqlabel{eq:gg}
    \end{alignat}
    (meaning $g_i(U(x))\leq \tilde g_i(x)$ for all $i$, where $g=(g_1,...,g_k)$, $\tilde g=(\tilde g_1,...,\tilde g_k)$),
    then we can find a version $\tilde U$ of $U$, with values in $K$ \emph{everywhere}, so that 
    \begin{alignat}{1}
        g(\tilde U) &\leq \tilde g \quad\text{for \emph{all} $x\in\Omega$.}\myeqlabel{eq:ggg}
    \end{alignat}
\end{lemma}
\begin{proof}
    We immediately modify $U$, on a set of measure $0$, to have values in $K$ everywhere.
    
    Let $E=\{x~|~g(U(x)) \leq \tilde g(x) \}$. Then $\complement E$ has measure zero,
    so every $x\in\complement E$ is the limit of a sequence $(x_n)$ in $E$.
    $(U(x_n))\subset K$ which is compact, so we may choose a subsequence $(x_n')$ so that $(U(x_n'))$ converges as well. 
    Define $\tilde U(x):=\lim(U(x_n'))\in K$ (we use \emph{one} subsequence for each $x$, as the limit for others may be different
    of course). Now
    $$ g(\tilde U(x))\leftarrow g(U(x_n))\leq\tilde g(x_n)\rightarrow\tilde g(x) ,$$
    so \myeqref{eq:ggg} is satisfied.
\end{proof}

\end{document}